\documentclass[a4paper,12pt]{amsart}
\usepackage{booktabs}
\usepackage[english]{babel}
\usepackage{iftex}
\ifPDFTeX
  \usepackage[T1]{fontenc}
  \usepackage{tgpagella}
  \makeatletter
  \def\qpl@set@spacing{%
    \fontdimen2\font=.25em
    \fontdimen3\font=.2em
    \fontdimen4\font=.1em}
  \makeatother
\else
  \usepackage[no-math]{fontspec}
\fi

\usepackage{amsmath,amsthm,amssymb,amsxtra,calligra,mathrsfs}
\usepackage{graphicx}
\usepackage{tikz}
\usepackage{tikz-cd} 
\usepackage{xcolor}
\usepackage{upgreek}
\usepackage{mathtools}
\usepackage{extarrows}
\usepackage{faktor}
\usepackage{tensor}
\usepackage{comment}
\usepackage{enumitem}
\usepackage{hyperref}
\hypersetup{
    colorlinks=true, 
    citecolor=blue,
    linkcolor=magenta,
	bookmarksnumbered=true,
}

\usepackage{geometry}
\usepackage[all,cmtip]{xy}

\newcommand{\C}{\mathbb C}
\newcommand{\B}{\mathbb B}
\newcommand{\Sym}{\operatorname{Sym}}
\newcommand{\Span}{\operatorname{span}}
\newcommand{\Aut}{\operatorname{Aut}}
\newcommand{\Gr}{\operatorname{Gr}}
\newcommand{\rank}{\operatorname{rank}}
\newcommand{\tr}{\operatorname{tr}}
\newcommand{\im}{\operatorname{im}}
\newcommand{\id}{\mathrm{id}}
\newcommand{\supp}{\operatorname{supp}}

\newcommand{\cH}{\mathcal H}

\newcommand{\HS}{\mathrm{HS}}

\newcommand{\ip}[2]{\left\langle#1,#2\right\rangle}
\newcommand{\w}{\wedge}

\newtheorem{theorem}{Theorem}[section]
\newtheorem{lemma}[theorem]{Lemma}
\newtheorem{proposition}[theorem]{Proposition}

\newtheorem{definition}[theorem]{Definition}
\theoremstyle{remark}\newtheorem{remark}[theorem]{Remark}
\numberwithin{equation}{section}

\title[Mok's uniqueness problem]{On Mok's uniqueness problem of holomorphic isometries from complex unit balls of maximal dimension into classical domains}
\author{Cong Ding, Yuan Yuan}
\date{}

\begin{document}

\begin{abstract}
We prove that every holomorphic isometric embedding, with respect to the Bergman metrics, from a complex unit ball of maximal dimension into an irreducible classical bounded symmetric domain of rank at least two that is not biholomorphic to a Type IV domain is equivalent to Mok’s embedding. This completely resolves a problem of Mok. 
\end{abstract}
\maketitle

\section{Introduction}

The study of holomorphic isometric embeddings has its origins in the work
of Bochner \cite{B47} on Hermitian geometry and Calabi \cite{C53} on
isometric embeddings of complex manifolds. Calabi introduced the diastasis
and established fundamental results on the existence, rigidity, and
analytic continuation of K\"ahler immersions into complex space forms.
In particular, a full K\"ahler immersion into a fixed complex space form
is determined, up to an ambient holomorphic isometry, by the induced
metric. 

Interest in this problem was renewed by Clozel--Ullmo \cite{CU03}, 
who reduced the problem of characterizing algebraic correspondences commuting with certain Hecke correspondences to rigidity questions for germs of holomorphic isometries from a bounded symmetric domain into Cartesian products of bounded symmetric
domains with respect to Bergman metrics. They showed that  Hermitian metric rigidity
theory of Mok \cite{M87} implies total geodesy when the source is irreducible of
rank at least two. Thus, for an irreducible source, nonstandard
holomorphic isometries can occur only when the source is a complex unit ball.
Mok \cite{M12} established algebraic extension and global continuation
for germs of holomorphic isometries between bounded symmetric domains,
allowing positive normalizing constants for the Bergman metrics. Mok also
constructed non-totally-geodesic isometries of the Poincar\'e disk into
polydisks \cite{M12}, disproving the corresponding total-geodesy conjecture of
Clozel--Ullmo. The reader may refer to \cite{MN09, N10, C16, C17} for related results on holomorphic isometries from the  Poincar\'e disk into
polydisks. On the other hand, if the complex dimension of the source complex unit ball is at least 2, it is recently proved in \cite{X26, YZ26} independently that the holomorphic isometries from the complex unit ball into a bounded symmetric domain must be componentwisely isometric,  generalizing earlier results in \cite{N11, YZ12, XY20, X22}. 
The problem is thus reduced to studying holomorphic isometries from the complex unit ball into irreducible bounded symmetric domains. 

Mok had established the following fundamental existence theorem \cite{M16} by considering  
the intersection of the bounded symmetric domain with the cone of minimal rational curves through a
regular boundary point. 

\begin{theorem}[Mok]\label{thm:Mok-ball-construction}
Let $D$ be an irreducible bounded symmetric domain of rank at least two,
and let $S_D$ be its compact dual. If $\delta$ is the positive generator
of $H^2(S_D,\mathbb Z)\cong\mathbb Z$, write
$c_1(S_D)=(d+1)\delta$. Then there exists a nonstandard proper holomorphic isometric embedding $F: \mathbb{B}^d \to D$
satisfying $F^*\omega_{D}=\omega_{\mathbb{B}^d}$. 
\end{theorem}

Let us recall Mok's construction briefly. Let $D\Subset \mathbb C^n \subset S$  be the inclusions incorporating both the
Borel embedding of $D$ into $S$, the compact dual manifold of $D$, and the Harish-Chandra realization $D\Subset \mathbb C^n$ of $D$. Let $\mathcal C_{x}(S)\subset T_x(S)$ be the variety of minimal rational tangents (alias. VMRTs) of $S$ at $x\in S$, or equivalently, the variety of tangents of minimal rational curves on $S$ passing through $x\in S$. Define 
\[\mathcal{V}_x:=\{\ell, \ell \ \text{is a minimal rational curve on} \ S \ \text{through} \ x\}\]
to be the full cone swept by minimal rational curves on $S$ through $x$. Choose a smooth boundary point $q\in \partial D$ and define $V_q:=\mathcal{V}_q \cap D$. Then $V_q \subset D$ is the image of the holomorphic isometric embedding in Theorem \ref{thm:Mok-ball-construction}.

This nonstandard holomorphic isometric embedding is called the Mok embedding.
Write $g_D$ for the genus of $D$ and set $d=g_D-1$.
Mok's dimension bound in \cite{M16} states that $d$ is the maximal
ball dimension for holomorphic isometries when the invariant metrics
are normalized so that minimal disks have Gaussian curvature $-2$.
We use this conventional maximal dimension throughout. At this dimension,
the source and target Bergman metrics are the same multiple, namely
$g_D=d+1$, of these normalized metrics. Thus at the maximal dimension 
the normalized and unrescaled Bergman isometry conditions agree.

Subsequent work clarified the structure of such embeddings.
Chan--Mok~\cite{CM17} showed that the images of holomorphic isometries
for the normalized metrics arise as irreducible components of linear
sections of the minimal projective embedding of the compact dual.
Upmeier--Wang--Zhang~\cite{UWZ19} provided a Jordan-theoretic construction of
Mok embeddings and characterized them by the existence of a nonzero
tangent direction annihilating the second fundamental form.
Xiao--Yuan~\cite{XY18} supplied explicit formulas for the classical
domains. For Type~IV domains, the classification results of
Chan--Mok~\cite{CM17}, Upmeier--Wang--Zhang~\cite{UWZ19}, and
Xiao--Yuan~\cite{XY20} exhibit inequivalent isometries 
even at the maximal ball dimension. 
The uniqueness then becomes a natural question for holomorphic isometric embeddings from the complex unit ball with maximal dimension into an irreducible  bounded symmetric domain that is not biholomorphic to a Type IV classical domain, known as Mok's problem (cf. Problem 2 in \cite{M26} for the explicit formulation). 
Yang \cite{Y17} proved uniqueness for $D^I_{2,m}$ with $m\ge3$,
$D^I_{3,m}$ with $m\ge3$, and the exceptional domains of types $E_6$
and $E_7$. The purpose of this paper is to establish the following uniqueness result  for all
the remaining classical domains, thereby providing, together with
Yang's results, the affirmative solution to Mok's problem.  The interested reader may also refer to related results such as \cite{M02, M11, MN12, HY14, Y19, FHX20, MW26, DY26}.

\begin{theorem}\label{thm:classical}
Let $D$ be an irreducible classical bounded symmetric domain of rank at
least two which is not biholomorphic to a Type~IV domain, and put
$d=g_D-1$. Every holomorphic map $F:\mathbb B^d\to D$ satisfying
$F^*\omega_D=\omega_{\mathbb B^d}$ is equivalent to the Mok embedding,
where $\omega_D$ and $\omega_{\mathbb B^d}$ are the K\"ahler forms of the Bergman metrics.
\end{theorem}

We use the following realizations of irreducible classical bounded symmetric domains other than Type IV domains.
\begin{align*}
D^I_{n,m}&=\{Z\in M_{n,m}(\C):I_n-ZZ^*>0\} {\rm~ with ~} n \leq m,\\
D^{II}_n&=\{Z\in M_n(\C):Z^t=-Z,\ I_n-ZZ^*>0\},\\
D^{III}_n&=\{Z\in M_n(\C):Z^t=Z,\ I_n-ZZ^*>0\},
\end{align*}
where $M_{n,m}(\C)$ denotes the complex $n\times m$ matrices, $M_n(\C)=M_{n,n}(\C)$, 
$Z^t$ is the transpose, $Z^*=\bar Z^t$, and $H>0$
means that the Hermitian matrix $H$ is positive definite.
By the following identification of classical domains: $D^{I}_{n, m} \cong D^{I}_{m, n}$; $D^I_{1,1}\cong D^{II}_2 \cong  D^{III}_1\cong \Delta=\B^1$; $D^{II}_3 \cong D^I_{3,1} \cong \mathbb{B}^3$; $D^{III}_2 \cong D^{IV}_3$; $D^I_{2,2} \cong D^{IV}_4$ and $D^{II}_4 \cong D^{IV}_6$, the target in the theorem reduces to Types I with
$2\le n\le m$ and $(n,m)\ne(2,2)$, to Type II with $n\ge5$, and to Type III with
$n\ge3$. 
Although the rank-two and rank-three Type I cases were already
covered in \cite{Y17}, we include them in a unified proof. 
Note that when $D$ is the classical domain $D^{I}_{n, m}, D^{II}_n$ or $D^{III}_n$, $d=n+m-1, 2n-3$ or $n$ respectively. 
The Bergman kernel on the diagonal is denoted by $K_D(Z,Z)$, and
$\omega_D=\sqrt{-1}\,\partial\bar\partial\log K_D(Z,Z)$ is the
 Bergman form. 
 Two maps $F_1,F_2:\B^d\to D$ are equivalent if
$F_2=\tau\circ F_1\circ\sigma$ for $\sigma\in\Aut(\B^d)$ and
$\tau\in\Aut(D)$.
Let \(e_1,\ldots,e_n\) be the standard basis of \(\mathbb C^n\) and the exterior product \(\Lambda^2\mathbb C^n\) has basis
$\{e_i\wedge e_j: 1 \leq i<j \leq n\}$.  
Therefore there is an identification between  the complex vector space of skew-symmetric matrices and \(\Lambda^2\mathbb C^n\)  given by the linear isomorphism
$Z  \mapsto\sum_{1 \leq i<j \leq n} Z_{ij}\,e_i\wedge e_j.$ 

\medskip

By the above realization, the next theorem
is what we will prove, which is equivalent to  Theorem \ref{thm:classical}.

\begin{theorem}\label{thm:main}
Let $D$ be an irreducible classical bounded symmetric domain of rank at least two that is not
biholomorphic to a Type IV domain and 
$F:\B^d\to D$ be a holomorphic map from the maximal dimensional ball $\B^d$ to $D$
 satisfying $F^*\omega_D=\omega_{\B^d}$. 
\begin{enumerate}[label=\textup{(\roman*)}]
\item If $D = D^I_{n,m} $ with $2\le n\le m$ and $(n,m)\ne(2,2)$, then
$d=n+m-1$ and $F$ is equivalent to
\begin{equation}\label{eq:model-I}
M^I(a,x,y)=\begin{pmatrix}a&x\\ y&\dfrac{yx}{1+a}\end{pmatrix},
\end{equation}
where $x\in M_{1,m-1}(\C)$ and $y\in M_{n-1,1}(\C)$ satisfying $|a|^2+\|x\|^2+\|y\|^2<1$.

\item If $D = D^{II}_{n} $ with $n\ge5$, then $d=2n-3$ and $F$ is equivalent to
\begin{equation}\label{eq:model-II}
M^{II}(s,u,v)=s e_1\w e_2+e_1\w u+e_2\w v
                  -\frac{u\w v}{1-s},
\end{equation}
where $u,v\in\Span(e_3,\ldots,e_n)$ and
$|s|^2+\|u\|^2+\|v\|^2<1$, by the identification with $\Lambda^2\mathbb C^n$. 

\item  If $D = D^{III}_{n} $ with $n\ge3$, then $d=n$ and $F$ is equivalent to 
\begin{equation}\label{eq:model-III}
M^{III}(a,\zeta)=
\begin{pmatrix}a& \frac{\zeta^t}{\sqrt2}\\
\frac{\zeta}{\sqrt2}&\dfrac{\zeta\zeta^t}{2(1+a)}\end{pmatrix},
\end{equation}
where $\zeta  \in M_{n-1,1}(\C)$ with $|a|^2+\|\zeta\|^2<1$. 
\end{enumerate}
\end{theorem}

\begin{remark}
These explicit maps $M^I, M^{II}, M^{III}$ have appeared in \cite{XY18} (cf. Equations (43),(51),(54)). 
\end{remark}

\begin{remark}
The same proof strategy is expected to extend to the exceptional domains of types \(E_6\) and \(E_7\) combining with the established tangent plane rigidity results of Yang. 
We do not carry out this extension here, since uniqueness in the exceptional cases is already obtained by Yang \cite{Y17}.
\end{remark}

We briefly discuss the main ideas in the proof. The first key input is a rescaling argument. By approaching to a regular boundary point of the target domain, we show that the tangent subspaces of  equivalent isometries approach the
tangent subspace of the model embedding. 
 The second key input is a local rigidity phenomena for tangent subspaces, meaning 
a sufficiently
close tangent subspace arising from an actual isometry is already standard
up to target automorphisms. The local rigidity follows from the jet identities derived from the isometry condition. In particular, third order jet is necessary for the Type III case.
Once the tangent subspace is normalized, the Bergman kernel
identity determines the holomorphic map explicitly, up to source and target automorphisms.
Since the whole procedure stays within the equivalence class of the model embedding, the holomorphic isometry is indeed unique. Namely, every holomorphic isometry is equivalent to the Mok embedding.
This overall strategy has already appeared in the earlier papers by Mok and Yang \cite{Y17, M18, M26}. 
In particular,  the rescaling method has played a crucial role in \cite{M10, M19, CM22} for the holomorphic isometries related problems
and the use of infinitesimal deformation of tangent subspaces after degeneration gives the key argument in Yang’s treatment of the exceptional domain of type \(E_7\) \cite{Y17}.
Moreover, our reconstruction argument is closely related to the characterization of the Mok embedding by Upmeier–Wang–Zhang \cite{UWZ19}.

This problem has an interesting natural connection with the rigidity problem of Schubert varieties. Recall that a Schubert variety $Z_w\subset S$ on a compact Hermitian symmetric space $S=G/P$ is the one for which the $G$-orbit of $G·[Z_w]$ in $\operatorname{Chow}(S)$ is projective. The Schubert variety $Z_w$ is said to be \textit{Schur rigid} if for any positive integer $r$, any subvariety $S_0\subset S$ with homology class $r[Z_w]$ must be of the form $g_1Z_w+\cdots+g_rZ_w(g_1,\cdots,g_r\in G)$. Walters and Bryant \cite{Bry00,Wa97} studied this problem by transforming it to the problem on the integral varieties of the Schur differential systems $\mathcal{R}_w$. Then $Z_w$ is Schur rigid if and only if any integral variety of $\mathcal R_w$ is the sum of $G$-translates of $Z_w$. The differential system $R_w$ contains the differential system $\mathcal B_w$ with the fiber consisting of the tangent space of all the Schubert varieties of type $w$ passing through the given point. We say that the Schubert variety $Z_w\subset S$ is
\textit{Schubert rigid} if any irreducible integral variety of $\mathcal B_w$ is a $G$-translate of $Z_w$. Schur rigidity then implies Schubert rigidity.  In our particular case, the cone swept by minimal rational curves $\mathcal V_x \subset S$ is a singular Schubert variety with isolated singularity $x$. In our proof, the final step to reconstruct the complex unit ball after fixing the tangent subspace can actually follows from the Schubert rigidity of $\mathcal V_x$ (see also \cite{M18f} for a proof via the sub-VMRT theory). Schur rigidity of singular Schubert varieties on compact Hermitian symmetric spaces has been completely settled in \cite{RT12}. We know $\mathcal V_x \subset S$ is Schur rigid if $S$ is not a hyperquadric or a rank 2 Grassmannian. This also provides a conceptual reason why Theorem \ref{thm:classical} fails for the Type IV domain, whose compact dual is a hyperquadric. As the $D^I_{2,m}$ case cannot be done by the known Schubert rigidity result, we include a self-contained proof in Section \ref{sec:reconstruction} for the reconstruction of the ball. We are also curious about the further interplay between Schubert varieties and cycles in bounded symmetric domains.

\medskip
\noindent\textbf{Notation and Conventions}
We fix the notation and conventions . 
Throughout the paper, $d$ denotes the complex dimension of the source ball.
The letters $n,m$ are reserved for dimensions: Type I uses $n\times m$
matrices, while Types II and III use $n\times n$ matrices. Thus
$d=n+m-1$, $d=2n-3$, and $d=n$, respectively.

For an integer $k\ge1$, $I_k$ denotes the $k\times k$ identity matrix,
and $e_1,\ldots,e_k$ is the standard basis of $\C^k$. The source ball is
$\B^d=\{z\in\C^d:\sum_{j=1}^d|z_j|^2<1\}$. The notation $\Sym_n(\C)=\{X\in M_{n,n}(\C):X^t=X\}$.
The matrix norm $\|X\|_{\mathrm{op}}$ is the operator norm for the
standard Euclidean norms. All other matrix norms are the Hilbert--Schmidt
norms specified below, unless stated otherwise. Namely, 
on vectors use $\langle z,w\rangle=\sum_jz_j\bar w_j$; on rectangular
and symmetric matrices use $\langle X,Y\rangle=\tr(XY^*)$;
on skew-symmetric matrices use $\langle X,Y\rangle=\tfrac12\tr(XY^*)$, where
the identification  is again $X\leftrightarrow\sum_{i<j}X_{ij}e_i\wedge e_j$ and 
this identification is isometric between the space of skew-symmetric matrices and $\Lambda^2\mathbb C^n$ equipped with the induced Euclidean metric.
For symmetric matrices, the diagonal matrix units $E_{ii}$ together
with $\frac{(E_{ij}+E_{ji})}{\sqrt2}$ for $i<j$ are orthonormal. 

For a linear subspace $L$ of a Hermitian vector space, $L^\perp$ is its
orthogonal complement and $P_L$ is the orthogonal projection onto $L$.
For a linear map $T$, its adjoint is defined by
$\langle Tx,y\rangle=\langle x,T^*y\rangle$.
The notation $\Span$ always means complex linear span, and
$\Gr(k,V)$ is the Grassmannian of $k$-dimensional complex linear
subspaces of $V$. The notation $\Sym^2V$ means the symmetric tensor
square; we specify separately when it is represented by symmetric
matrices. The notation $\Lambda^k V$ means the $k$th exterior power, and
$\operatorname{Hom}(V,W)$ is the space of complex-linear maps from
$V$ to $W$. We write $\im T$, $\ker T$, $\rank T$, and $\tr T$
for image, kernel, rank, and trace. The symbols $I$ and $\id$ denote
identity operators on the space indicated by context. The group
$U(k)=\{V\in M_k(\C):V^*V=I_k\}$ is the unitary group, and
$\Aut(\Omega)$ is the group of biholomorphic maps of a domain $\Omega$
onto itself. For a nonzero vector $v$, $[v]\in\mathbb P(V)$ denotes
its complex line. 

For a holomorphic map fixing the origin, write
\[
F(z)=Az+\sum_{j\ge2}F_j(z),\quad A=\mathrm{d}F_0,
\quad F_j(tz)=t^jF_j(z).
\]
Here $\mathrm{d}F_0=F'(0)$ is the complex differential at the origin.
For a multi-index $\alpha=(\alpha_1,\ldots,\alpha_d)\in\mathbb Z_{\ge0}^d$,
write $|\alpha|=\sum_i\alpha_i$ and $z^\alpha=\prod_i z_i^{\alpha_i}$.
Thus $F_j$ is the sum of the terms with $|\alpha|=j$, including
the usual Taylor factorial in its coefficients. The notation $O(2)$
means that the constant and linear Taylor terms vanish. In a polarized
identity, bidegree $(j,k)$ means homogeneous degree $j$ in $z$ and
degree $k$ in the independent conjugate variables $\bar w$.

\section{The  jet identities}\label{sec:common}

Fix one of the three cases in Theorem~\ref{thm:main}. Let $\mathcal M$ be
the ambient complex vector space: $M_{n,m}(\C)$ in Type I,
$\Lambda^2\C^n$ in Type II, and $\Sym_n(\C)$ in Type III.
 Write $N=\dim_\C \mathcal M$ for the target dimension,
$d$ for the source dimension, and $d^\#=N-d$ for the codimension of an
isometric tangent subspace.
Recall the following standard 
 homogeneous polynomial maps $Q_k:\mathcal M\to\mathcal H_k$ as follows (cf. \cite{BL14}). In Type I,
\[
\mathcal H_k=\operatorname{Hom}(\Lambda^k\C^m,\Lambda^k\C^n),
\quad [Q_k(Z)]_{I,J}=\det Z_{I,J},
\]
where $I=\{i_1<\cdots<i_k\}\subset\{1,\ldots,n\}$ and
$J=\{j_1<\cdots<j_k\}\subset\{1,\ldots,m\}$, and 
 $Z_{I,J}$ is the submatrix with these rows and columns.
Equip $\mathcal H_k$ with the Hilbert--Schmidt metric. Namely, the squared norm
is the sum of the squared moduli of all these minors.
In Type III we use the same construction with $m=n$ and restrict it
to symmetric matrices. We write $Q_k=C_k$ in these two cases with \(C_k(Z)\) being the \(k\)th compound matrix of \(Z\). 
 In Type II, 
\[
\mathcal H_k=\Lambda^{2k}\C^n,\quad
Q_k(Z)=\frac{Z^{\w k}}{k!}.
\]
More precisely,  
$Q_k(Z)=\frac1{k!}
\underbrace{Z\wedge Z\wedge\cdots\wedge Z}_{k\text{ factors}}$, for every 
$Z=\sum_{i<j}Z_{ij}\,e_i\wedge e_j\in\Lambda^2\mathbb C^n$.
In fact, the coordinates of \(Q_k(Z)\) are exactly the Pfaffians of the principal \(2k\times2k\) submatrices (cf. \cite{M09} for the detailed construction, which is originally due to Chevalley \cite{C97}).
Equip 
$\Lambda^{2k}\mathbb C^n$ with the induced Euclidean metric.
 Note that each $Q_k$ has degree $k$, and $Q_1$ identifies
$\mathcal M$ isometrically with its entry coordinates. Let $Q_0=1$ and the generic-norm can be written as 
\begin{equation}\label{eq:generic-common}
\mathcal N(Z,W)=\sum_{k=0}^{\ell}(-1)^k
                   \ip{Q_k(Z)}{Q_k(W)},
\end{equation}
where $\ell=n$ in Types I and III and
$\ell=\lfloor n/2\rfloor$ in Type II (cf. \cite{CM17}). Denote the genus by $g$, that is the exponent in the Bergman kernel formula. The dimensions and genus are given in the following table. 

\begin{center}
\begin{tabular}{@{}lcccc@{}}
\toprule
Type & $N$ & $d$ & $d^\#$ & $g=d+1$\\
\midrule
I & $nm$ & $n+m-1$ & $(n-1)(m-1)$ & $n+m$\\
II & $\binom n2$ & $2n-3$ & $\binom{n-2}{2}$ & $2n-2$\\
III & $n(n+1)/2$ & $n$ & $n(n-1)/2$ & $n+1$\\
\bottomrule
\end{tabular}
\end{center}

Let $D$ be one of Type I, II and III domains and $Z_0 \in D$. 
Then 
\begin{equation}\label{eq:transvection-common}
\tau_{Z_0}(Z)=(I_n-{Z_0}{Z_0}^*)^{-\frac 12}(Z-{Z_0})(I-{Z_0}^*Z)^{-1}(I_m-{Z_0}^*{Z_0})^{\frac 12}
\end{equation}
is an automorphism $D$, sends ${Z_0}$ to 0, with the inverse
$\tau_{-{Z_0}}$, and satisfies
\begin{equation}\label{eq:derivative-common}
(\mathrm{d}\tau_{Z_0})_{Z_0}(X)=(I_n-{Z_0}{Z_0}^*)^{-\frac 12}X(I_m-{Z_0}^*{Z_0})^{-\frac 12}.
\end{equation}
Here the square roots are the unique positive definite Hermitian square roots

The Bergman kernel function of $D$ is given by 
\begin{equation}\label{eq:kernels-common}
K_D(Z,Z)=K_D(0,0)\mathcal N(Z,Z)^{-g},
\end{equation}
with $g=d+1$. 
Let $F: \mathbb{B}^{d} \to D$ be a holomorphic map satisfying $F^*\omega_D=\omega_{\mathbb{B}^d}$. 
Without loss of generality, by composing the automorphism of $\mathbb{B}^d$ and $D$. assume 
 $F(0)=0$. 
It follows from
\eqref{eq:kernels-common}
and the standard reduction (cf. \cite{CU03, M12}) that
$$\mathcal N(F(z),F(z))=1- |z|^2,$$
for $z\in \B^d$. By polarization, we have 
\begin{equation}\label{eq:kernel-common}
\mathcal N(F(z),F(w))=1-\ip zw, 
\end{equation}
for $z,w\in\B^d$.

We now formulate an equivalent statement of Calabi's local rigidity theorem \cite{C53}.

\begin{lemma}\label{lem:gram-common}
Let $\Omega\subset\C^d$ be a connected open neighborhood of $0$.
For \(\nu=1,2\), let \((V_\nu,\langle\cdot,\cdot\rangle_\nu)\) be a finite-dimensional Hermitian vector space, and let
$H_\nu:\Omega\rightarrow V_\nu$
be holomorphic. Define
$V^\#_\nu:=\operatorname{span}_{\mathbb C}\{H_\nu(z):z\in\Omega\}\subseteq V_\nu.$ 
Assume
\begin{equation}\label{alice}
\langle H_1(z),H_1(w)\rangle_1 = \langle H_2(z),H_2(w)\rangle_2
\end{equation}
for all $z,w \in\Omega$. 
Then there exists a unique complex linear bijective isometry
$T: V^\#_1\rightarrow V^\#_2$
such that
\[
T H_1(z)=H_2(z),
\]
for all $z\in\Omega$.
\end{lemma}

\begin{proof}
Let $m_\nu=\dim_{\mathbb C}V_\nu, N=\max\{m_1,m_2\}.$
By choosing orthonormal bases and adding zero coordinates if necessary, define complex-linear isometric embeddings
$J_\nu:V_\nu\hookrightarrow\mathbb C^N$ satisfying
$\langle J_\nu x, J_\nu y\rangle =\langle x,y\rangle_\nu$, where $\langle \cdot, \cdot \rangle$ on $\C^N$ is the standard Euclidean inner product. 
Setting $\widetilde H_\nu=J_\nu\circ H_\nu$, then 
(\ref{alice}) implies 
$ \|\widetilde H_1(z)\|^2 =\|\widetilde H_2(z)\|^2.$
It then follows from Calabi's rigidity theorem that there exists a unitary matrix  \(U\in U(N)\) satisfying
$U J_1H_1(z)=J_2H_2(z).$
It follows that $ T=(J_2|_{V^\#_2})^{-1}\circ U\circ(J_1|_{V^\#_1}): V^\#_1\rightarrow V^\#_2$ is the desired linear bijective isometry. 
\end{proof}

The next lemma is due to  \cite{CM17, UWZ19} and we present a  re-formulation here. 

\begin{lemma}[Chan-Mok, Upmeier–Wang–Zhang]\label{lem:quadratic-common}
Let $F: \mathbb{B}^{d} \to D$ be a holomorphic map satisfying $F^*\omega_D=\omega_{\B^d}$. 
Write $F(z)=Az+F_2(z)+F_3(z)+\cdots$, where $F_j$ is homogeneous
of degree $j$. Let $E=A\C^d$ be the tangent space to  \(F(\B^d)\) in $\mathcal M$ at \(F(0)=0\) and let \(R=\operatorname{span}_{\mathbb C}Q_2(E)\subseteq\mathcal H_2\).  Then
\begin{equation}\label{eq:normal-common}
A^*F(z)=z,\quad A^*A=I_d,\quad F_j(z)\in E^\perp\quad(j\ge2).
\end{equation}
Moreover, there is an isometric embedding 
$T:R\to E^\perp$ with
\begin{equation}\label{eq:quadratic-common}
F_2(z)=TQ_2(Az),\quad \dim R\le d^\#.
\end{equation}
\end{lemma}
\begin{proof}
It follows from $E=\operatorname{im} A$ 
that \(\ker A^*=E^\perp.\)
Because $F(0)=0$, the smallest degree of $Q_k(F(w))$ is $k$.
It follows from (\ref{eq:generic-common}) and (\ref{eq:kernel-common}) that antiholomorphic degree one
can come only from $Q_1(F)=F$. It yields
$\langle F(z),Aw\rangle=\langle z,w\rangle$ for every $w$,
which is exactly $A^*F(z)=z$. 
Namely, $A^*Az+\sum_{j\ge2}A^*F_j(z)=z.$ Comparing the linear terms in $z$ 
 yields
$A^*A=I$.
Comparing the higher order terms in $z$ 
 yields $A^*F_j(z)=0,$ so $F_j(z)\in E^\perp$ for all $j \ge 2$. 

 After subtracting the
linear terms, (\ref{eq:kernel-common}) reads
\begin{equation}\label{eq:feature-common}
\ip{H(z)}{H(w)}+
\sum_{\substack{k\ge3\\k\text{ odd}}}\ip{Q_k(F(z))}{Q_k(F(w))}
=\sum_{\substack{k\ge2\\k\text{ even}}}\ip{Q_k(F(z))}{Q_k(F(w))}, 
\end{equation}
where $H(z)=F(z)-Az \in E^\perp$. By
comparing terms of bidegree \((2,2)\): degree \(2\) in \(z\) and degree \(2\) in \(\overline w\), it follows that 
$$\langle F_2(z),F_2(w)\rangle_{\mathcal M} = \langle Q_2(Az),Q_2(Aw)\rangle_{\mathcal H_2}.$$
By Lemma~\ref{lem:gram-common}, there exists a isometric embedding $T: R \to \operatorname{span}_{\C} F_2(z) \subset E^\perp$ such that $F_2(z)=TQ_2(Az)$, and $\dim R\le \dim E^\perp = d^\#$. 
\end{proof}
\begin{remark}
    The dimension bound of $R$ can be also obtained from the duality principle (cf. \cite{M26}).
\end{remark}

The next result is a slight generalization of Lemma 4.9 in \cite{UWZ19}.

\begin{proposition}\label{prop:cubic-common}
Let $\widetilde Q_2:\mathcal M\times \mathcal M\to\mathcal H_2$ be the symmetric complex-bilinear
polarization given by 
$\widetilde Q_2(X,Y)=Q_2(X+Y)-Q_2(X)-Q_2(Y)$.
If $\dim R=d^\#$, then the isometry $T:R\to E^\perp$ is onto, and the projections $P_R$ and $P_{R^\perp}$ on $\mathcal H_2$ satisfy 
\begin{align}
F_3(z)&=T P_R \widetilde Q_2(Az,F_2(z)),\label{eq:F3-common}\\
\ip{P_{R^\perp}\widetilde Q_2(Az,F_2(z))}{P_{R^\perp}\widetilde Q_2(Aw,F_2(w))}
&=\ip{Q_3(Az)}{Q_3(Aw)}.\label{eq:cubic-common}
\end{align}
\end{proposition}

\begin{proof}
By comparing terms of  bidegree $(3,2)$ in \eqref{eq:feature-common}, we get
\[
\ip{F_3(z)}{TQ_2(Aw)}=\ip{\widetilde Q_2(Az,F_2(z))}{Q_2(Aw)}.
\]
If $\dim R=d^\#$, then $T: R\to E^\perp$ is a unitary isomorphism with $T^{-1}=T^*$. 
The above identity implies 
$T^*F_3=P_R\widetilde Q_2(Az,F_2(z))$, proving \eqref{eq:F3-common}. 
By comparing terms of  bidegree $(3,3)$ in
\eqref{eq:feature-common}, we get 
\[
\langle F_3(z),F_3(w)\rangle+
\langle Q_3(Az),Q_3(Aw)\rangle
=\langle \widetilde Q_2(Az,F_2(z)),\widetilde Q_2(Aw,F_2(w))\rangle.
\]
Because \(T\) preserves inner products, (\ref{eq:F3-common}) yields
\[
\langle F_3(z),F_3(w)\rangle
=
\langle P_R \widetilde Q_2(Az,F_2(z)),P_R \widetilde Q_2(Aw, F_2(w))\rangle.
\]
Using the orthogonal decomposition \(\mathcal H_2=R\oplus R^\perp\) and applying $P_R, P_{R^\perp}$ to $\widetilde Q_2(Az,F_2(z))$, we obtain
\[
\langle Q_3(Az),Q_3(Aw)\rangle=
\langle P_{R^\perp}\widetilde Q_2(Az,F_2(z)),P_{R^\perp}\widetilde Q_2(Aw,F_2(w))\rangle.
\]
\end{proof}

\section{Rescaling and convergence of tangent planes}\label{sec:boundary}

Let $A_0=\mathrm dM_0: \C^d\to \mathcal M$ be the differential at \(0\) of the corresponding map \(M\) in Theorem \ref{thm:main}. The straightforward calculation yields 
\begin{equation}\label{eq:standard-inclusions}
\begin{aligned}
A_0^I(a,x,y)&=\begin{pmatrix}a&x\\y&0\end{pmatrix},\\
A_0^{II}(s,u,v)&=s e_1\w e_2+e_1\w u+e_2\w v,\\
A_0^{III}(a,\zeta)&=\begin{pmatrix}a&\frac{\zeta^t}{\sqrt2}\\
\frac{\zeta}{\sqrt2}&0\end{pmatrix}.
\end{aligned}
\end{equation}
We omit the superscript when the type is fixed. Each map preserves
norms by our entry and exterior-power conventions.
The image   
$E_0=A_0\C^d=T_0\bigl(M(\mathbb B^d)\bigr)$ of $A_0$ is given by 
\begin{equation}\label{eq:standard-planes}
\begin{array}{ll}
\mathrm{Type~ I}:&E_0=\left\{\begin{pmatrix}a&x\\y&0\end{pmatrix}\right\},\\[2mm]
\mathrm{Type~ II}:&E_0=\left\{
s\,e_1\wedge e_2+e_1\wedge u+e_2\wedge v
\right\}= P\w\C^n,\quad P=\Span(e_1,e_2),\\[1mm]
\mathrm{Type~ III}:&E_0=\left\{\begin{pmatrix}a&b^t\\b&0\end{pmatrix}\right\}.
\end{array}
\end{equation}
Their dimensions are precisely the corresponding values of $d$.

Let $F: \B^d \to D$ be the holomorphic map satisfying $F^*\omega_D=\omega_{\B^d}$. Mok proved that $F$ is algebraic \cite{M12}. As a consequence, there exists $b \in \partial\B^d$, such that $F$ extends holomorphically across $b$. Moreover, Mok showed that $F$ 
exits $D$ through $\operatorname{Reg}(\partial D)$ \cite{M16}. Therefore, at ${Z_\partial}=F(b)$, $I-{Z_\partial}{Z_\partial}^*$ has kernel dimension one
if $D$ is of Types I or Type III, and dimension two if $D$ is of Type II.
 We apply the following rescaling argument. 
 There exists a unitary transform  \(U\in U(d)\subset \operatorname{Aut}(\B^d)\) such that \(Ue_1=b\). 
For \(0<r<1\), define $ a_r(z_1,z') = \left( \frac{r+z_1}{1+rz_1}, \frac{\sqrt{1-r^2}\,z'}{1+rz_1} \right) \in \operatorname{Aut}(\B^d)$. Denote $ Z_r=F(rb) \in D$ and let $\tau_{Z_r} \in \operatorname{Aut}(D)$ be defined as in (\ref{eq:transvection-common}) satisfying $\tau_{Z_r}(Z_r)=0$.

\begin{proposition}\label{prop:boundary-scaling}
There exists  a unitary linear automorphism \(\tau \in\operatorname{Aut}(D)\) independent of $r \in (0, 1)$ such that 
$F_r=\tau \circ\tau_{Z_r}\circ F\circ U\circ a_r$ satisfies  $F_r(0)=0$ and $F_r^*\omega_D=\omega_{\mathbb B^d}$. 
Moreover, $\lim_{r\to 1^-}\|P_{E_r}-P_{E_0}\|_{\mathrm{op}}=0,$ where $   E_r=\mathrm d(F_r)_0(\mathbb C^d)\subset\mathcal M$ and  the projections and operator norm are taken with respect to the corresponding Hermitian inner product on \(\mathcal M\).
\end{proposition}

\begin{proof}
It is a direct calculation that $F_r(0)=0$ and $F_r^*\omega_D=\omega_{\mathbb B^d}$ since $F_r$ is equivalent to $F$. 
Let
 $u_{r,1}=(1-r^2) e_1=\mathrm d(a_r)_0(e_1), 
u_{r,j}=\sqrt{1-r^2} e_j=\mathrm d(a_r)_0(e_j) $, for $2\le j\le d$ be an orthogonal basis of the tangent space of $\B^d$ at $r e_1$ with respect to the Bergman metric. 
Let $X_{r,j}=\mathrm{d}F_{(r b)}(U u_{r,j})$.
Since $F$ is holomorphic near $b$, then  $\mathrm{d}F$ is bounded near $b$ and thus
$\|X_{r,j}\|\le C\sqrt{1-r^2}\to0$ as $r \to 1^-$.
Let  $Y_{r,j}=(\mathrm{d}\tau_{Z_r})_{Z_r}X_{r,j}=(I-Z_rZ_r^*)^{-\frac 12}X_{r,j}(I-Z^*_rZ_r)^{-\frac 12}$ and \(\widetilde E_r =\operatorname{span}_{\mathbb C} \{Y_{r,1}, \cdots, Y_{r,d}\} \). 
It is easy to verify that $Y_{r, 1}, \cdots, Y_{r, d}$ is an orthonormal basis of $\widetilde E_r$ with respect to the corresponding inner product.

In Type I,
write the singular values of \(Z_r\in M_{n,m}(\mathbb C)\), with \(n\le m\), as
$s_1(r)\ge s_2(r)\ge\cdots\ge s_n(r)\ge0.$
The eigenvalues of \(Z_rZ_r^*\) are $s^2_1(r), \ldots, s^2_n(r)$ and those of \(Z_r^*Z_r\) are the same numbers, together with \(m-n\) additional zeros.
At the regular boundary value \(Z_\partial\),
$s_1(\partial)=1, s_2(\partial)<1$ since \(I-Z_\partial Z_\partial^*\ge0\) and $I-{Z_\partial}{Z_\partial}^*$ has kernel dimension one. 
As $Z_r\to Z_\partial$, by continuity of singular values, for \(r\) sufficiently close to \(1\), the largest singular value \(s_1(r)\) is simple and $s_2(r)$ satisfies $s_2(r) \leq 1-\delta$ for some $\delta\in (0, 1)$.
Define $P_r^L =\text{orthogonal projection of }\mathbb C^n \text{ onto } \ker\!\left(Z_rZ_r^*-s^2_1(r) I_n\right)$,
and $P_r^R =\text{orthogonal projection of }\mathbb C^m \text{ onto } \ker\!\left(Z_r^*Z_r-s^2_1(r) I_m\right)$.
It follows that 
$$\|(I-P_r^L)(I-Z_rZ_r^*)^{-\frac 12} \|_{\mathrm{op}}
\le\delta^{-\frac 12}$$ and 
$$\|(I-Z^*_rZ_r)^{-\frac 12}(I-P_r^R)\|_{\mathrm{op}} \le\delta^{-\frac 12}.$$
We similarly define $P^L =\text{orthogonal projection of }\mathbb C^n \text{ onto } \ker\!\left(Z_\partial Z_\partial^*- I_n\right)$,
and $P^R =\text{orthogonal projection of }\mathbb C^m \text{ onto } \ker\!\left(Z_\partial^*Z_\partial-  I_m\right)$.
For $X \in \mathcal M$, let $\Pi_r(X)=(I-P_r^L)X(I-P_r^R), \Pi(X)=(I-P^L)X(I-P^R)$ and $K=\ker ( \Pi)$.
It is obvious that $\Pi_r^2=\Pi_r$. Moreover, for any $X, Y \in \mathcal M$, 
\[
\begin{aligned}
\langle\Pi_r(X),Y\rangle
&=\operatorname{tr}((I-P_r^L) X (I-P_r^R) Y^*) =\operatorname{tr}(X (I-P_r^R) Y^* (I-P_r^L))\\
&=\operatorname{tr}\bigl(X((I-P_r^L)Y (I-P_r^R) )^*\bigr)=\langle X,\Pi_r(Y)\rangle.
\end{aligned}
\]
It follows that \(\Pi_r^*=\Pi_r.\) Therefore, 
 $\Pi_r$ is an orthogonal projections on $\mathcal M$ and so is $\Pi$ by the same argument.  Because \(P^L,P^R\) have rank one, \(K\) has dimension \(n+m-1=d\) and $P_K=I-\Pi$.
Writing $$(\Pi_r-\Pi)(X) =(P^L-P_r^L)X(I-P_r^R) +(I-P^L)X(P^R-P_r^R),$$ since  \(P_r^L\to P^L\) and \(P_r^R\to P^R,\) 
we have $$\|\Pi_r-\Pi\|_{\mathrm{op}} \le \|P_r^L-P^L\|_{\mathrm{op}} +\|P_r^R-P^R\|_{\mathrm{op}} \rightarrow0.$$
Moreover,
\begin{align*}
\|\Pi_r(Y_{r,j})\| &= \|(I-P_r^L)(I-Z_rZ_r^*)^{-\frac 12}X_{r,j}(I-Z^*_rZ_r)^{-\frac 12}(I-P_r^R) \| \\
&\le\delta^{-1/2}\|X_{r,j}\|\delta^{-1/2} \\
&\le C \delta^{-1} \sqrt{1-r^2}.
\end{align*}
Consequently, as $r \to 1^-$, 
\begin{align*}
\|\Pi(Y_{r,j})\|
\le
\|(\Pi-\Pi_r)Y_{r,j}\|+\|\Pi_r(Y_{r,j})\| \le
\|\Pi-\Pi_r\|_{\mathrm{op}}+C\delta^{-1} \sqrt{1-r^2}
\rightarrow0. 
\end{align*}
Since \(Y_{r,1},\ldots, Y_{r,d}\) are orthonormal, 
\(
P_{\widetilde E_r}(X)
=\sum_{j=1}^d
\langle X,Y_{r,j}\rangle\,Y_{r,j}
\)
for $X\in\mathcal M$. 
It then follows that, as $r \to 1^-$, 
\begin{align*}
\|P_{\widetilde E_r}-P_K\|^2
&=2d-2\operatorname{tr}(P_{\widetilde E_r}P_K)=2\sum_{j=1}^d
  \left(1-\|P_K Y_{r,j}\|^2\right)=2\sum_{j=1}^d\|\Pi(Y_{r,j})\|^2 \to 0.
\end{align*}
Finally, choose unit vectors \(u,v\) such that
$\operatorname{im}P^L=\mathbb C u, \operatorname{im}P^R=\mathbb C v,$
 complete them to orthonormal bases and form the unitary matrices
$A=(u,u_2,\ldots,u_n), B=(v,v_2,\ldots,v_m).$
Define $\tau(X)=A^*XB.$ It is easy to verify that $\tau \in \operatorname{Aut}(D)$ and  $\tau(K)=E_0$. Let $E_r=\tau(\widetilde E_r)$. 
It follows from $P_{E_r}=\tau P_{\widetilde E_r}\tau^{-1}, P_{E_0}=\tau P_K\tau^{-1}$ that
$$ \|P_{E_r}-P_{E_0}\|_{\mathrm{op}} = \|P_{\widetilde E_r}-P_K\|_{\mathrm{op}} \rightarrow0.$$

For Types II and III, the proof will be similar and the point out the difference. 
Let $P_r=$ orthogonal projection of $\mathbb C^n \text{ onto } \ker\!\left(Z_rZ_r^*-s^2_1(r) I_n\right)$. Then
$\operatorname{rank}P_r=2$ in Type II and $\operatorname{rank}P_r=1$ in Type III. 
As above, there exist
$\delta>0$ and $r_0<1$ such that every eigenvalue of
$Z_rZ_r^*$ on $\ker P_r$ is at most $1-\delta$
whenever $r_0<r<1$.
It follows that $(I-Z_rZ_r^*)^{-\frac 12}$ commutes with $P_r$ and
$$ \bigl\|(I-P_r)(I-Z_rZ_r^*)^{-\frac 12} \bigr\|_{\mathrm{op}} \le\delta^{-\frac 12}.$$
Define $\Pi_r(X)=(I-P_r)X(I-P_r)^t$ and $\Pi(X)=(I-P)X(I-P)^t$, where $P$ is the orthogonal projection onto $\ker(I-Z_\partial Z_\partial^*)$. 
It is easy to verify that the dimension of \(K=\ker\Pi\) is always $d$.  
Using 
$$ \Pi_r(Y_{r,j}) =\bigl[(I-P_r) (I-Z_rZ_r^*)^{-\frac 12} \bigr]  X_{r,j} \bigl[(I-P_r)(I-Z_rZ_r^*)^{-\frac 12}\bigr]^t,$$
we have $$\|\Pi_r(Y_{r,j})\| \le C\delta^{-1}\sqrt{1-r^2}
\rightarrow0.$$
Since $P_r\rightarrow P$ as $r \to 1^-$,  we have $\|\Pi_r-\Pi\|_{\mathrm{op}}  \rightarrow0.$
Therefore, $\|\Pi(Y_{r,j})\| \rightarrow0.$ This further implies $\|P_{\widetilde E_r}-P_K\|^2 \to 0.$
Take an orthonormal basis of  $\operatorname{im}P=\ker(I-Z_\partial Z_\partial^*)$. 
In Type II, the basis is \(u_1, u_2\)  and  in Type III, the basis is \(u_1\).
Complete them to orthonormal bases and form the unitary matrices $A= (u_1, u_2,..., u_n)$. 
Then $\tau(X)=A^*X\overline A$ is the desired element in $ \operatorname{Aut}(D)$ which yields $E_0=\tau(K)$, $E_r=\tau(\widetilde E_r)$ and $ \|P_{E_r}-P_{E_0}\|_{\mathrm{op}} \rightarrow0.$
This finishes the proof of Types II and Type III.
\end{proof}

\section{Local rigidity of the tangent plane}
\label{sec:exact-tangent}

We have shown $E_r=\mathrm{d}(F_r)_0(\C^d)\to E_0$ and $F_r(0)=0$.
We now prove that one fixed map in this family can be normalized to
have exactly the standard derivative. Let $A_0:\C^d\to \mathcal M$ be the
standard linear isometry defined in \eqref{eq:standard-inclusions};
thus $E_0=A_0(\C^d)$.

For any $d$-plane $E\subset \mathcal M$, define its quadratic rank by
\[
\mathfrak r(E)=\dim_\C\Span\{Q_2(X):X\in E\}.
\]

For Types I and III, $\dim R$ is also the dimension of the span of
the restricted quadratic minor polynomials. In Type II it is
$\dim\Span\{X\w Y:X,Y\in E\}$, by polarization of $X\w X$.

Every plane $E_r$ satisfies $\mathfrak r(E_r)\le d^\#$ by
Lemma~\ref{lem:quadratic-common}. In Types I and II we classify nearby
planes satisfying this bound. In Type III the bound alone permits
additional planes; the cubic identity in Proposition \ref{prop:cubic-common} excludes those
additional possibilities. 

\begin{proposition}
\label{prop:exact-tangent}
Let $F:\B^d\to D$ be a holomorphic map  satisfying $F^*\omega_D=\omega_{\B^d}$. 
Then there exist
$\sigma_0\in\Aut(\B^d)$ and $\tau_0\in\Aut(D)$ such that
\[
G=\tau_0\circ F\circ\sigma_0,
\quad G(0)=0,\quad \mathrm{d}G_0=A_0.
\]
\end{proposition}

We first prove the needed local assertion for each type, and then
complete the proof of this proposition in
Section~\ref{sec:exact-tangent-assembly}.
\subsection{Type I: the graph lemma and local tangent rigidity}
\label{I:section}

Let \(2\le n\le m\), with \((n,m)\ne(2,2)\), and set
\[
 s=n-1,\quad t=m-1,\quad
 d=n+m-1=1+s+t,\quad d^\#=nm-d=st.
\]
Thus \(s\ge1\) and \(t\ge2\). In the ambient space \(M_{n,m}(\C)\),
we use block coordinates
\[
 Z=\begin{pmatrix}a&x\\y&W\end{pmatrix},\quad
 \begin{gathered}
 a\in\C,\quad x\in M_{1,t}(\C),\\
 y\in M_{s,1}(\C),\quad W\in M_{s,t}(\C).
 \end{gathered}
\]
The cross plane and its coordinate projection are
\[
 E_0=\left\{\begin{pmatrix}a&x\\y&0\end{pmatrix}:a,x,y\right\},
\]
\[
 \begin{aligned}
 \pi:M_{n,m}(\C)&\rightarrow
       \C\oplus M_{1,t}(\C)\oplus M_{s,1}(\C),\\
 \pi\begin{pmatrix}a&x\\y&W\end{pmatrix}&=(a,x,y).
 \end{aligned}
\]

For a linear subspace \(E\subset M_{n,m}(\C)\), write
\[
 r(E)=\dim_\C\Span\{\text{two-by-two minors restricted to }E\},
\]
where the span is taken in \(\Sym^2(E^*)\), the space of homogeneous
quadratic polynomials on \(E\). The quadratic identity in Lemma \ref{lem:quadratic-common} implies
\(r(E)\le d^\#=st\) when \(E\) is the tangent subspace of a normalized
isometry. We will prove a local classification using only this
inequality and the invertibility of \(\pi|_E\).

\subsubsection{The graph lemma}

\begin{lemma}
\label{I:graph}
Suppose \(E\subset M_{s+1,t+1}(\C)\) has dimension \(1+s+t\),
\(\pi|_E\) is an isomorphism, and \(r(E)\le st\), where
\(s\ge1\) and \(t\ge2\). Then there are a column
\(u\in M_{s,1}(\C)\) and a row \(v\in M_{1,t}(\C)\) such that
\begin{equation}\label{I:graphformula}
 E=\left\{\begin{pmatrix}a&x\\y&ux+yv-auv\end{pmatrix}:a,x,y\right\}.
\end{equation}
Equivalently, for the nonzero column \(\widehat{u}=\binom1u\) and row
\(\widehat{v}=(1,v)\),
\begin{equation}\label{I:lineform}
 E=\widehat{u}\,M_{1,m}(\C)+M_{n,1}(\C)\,\widehat{v}.
\end{equation}
The first summand consists of all matrices with columns in \(\C \widehat{u}\);
the second consists of all matrices with rows in \(\C \widehat{v}\).
\end{lemma}

\begin{proof}
Because \(\pi|_E\) is a linear isomorphism, each triple \((a,x,y)\)
has a unique inverse image in \(E\). Thus
\[
 E=\left\{\begin{pmatrix}a&x\\y&W(a,x,y)\end{pmatrix}:a,x,y\right\},
 \quad W(a,x,y)=aB+\mathcal L(x)+\mathcal K(y),
\]
where \(B\in M_{s,t}(\C)\), and
\[
 \mathcal L:M_{1,t}(\C)\rightarrow M_{s,t}(\C),
 \quad
 \mathcal K:M_{s,1}(\C)\rightarrow M_{s,t}(\C)
\]
are linear maps. All polynomial comparisons below are in the independent
coordinates \(a,x_1,\ldots,x_t,y_1,\ldots,y_s\) on \(E\).

For \(1\le i\le s\), \(1\le j\le t\), the minor on rows
\(1,i+1\) and columns \(1,j+1\) is
\[
 P_{ij}=aW_{ij}-y_ix_j.
\]
Note that $W_{ij}$ are linear functions with respect to $(a,x,y)$ and $aW_{ij}$ have no mixed $yx$-part. These \(st\) polynomials are linearly independent. Because
\(r(E)\le st\), the \(P_{ij}\) form a basis for the span of all
restricted two-by-two minors. Every other minor is consequently a
constant linear combination of these corner minors.

Fix \(i\). For distinct \(j,k\), consider
\begin{equation}\label{equ:row minor}
    R_{i;jk}=x_jW_{ik}-x_kW_{ij},
\end{equation}
the minor on rows \(1,i+1\) and columns \(j+1,k+1\), with the
indicated order. No \(P_{ij}\) has a term depending quadratically on
\(x\) alone. Comparing such terms in \(R_{i;jk}\) yields
\begin{equation}\label{I:Cidentity}
 x_j(\mathcal L(x))_{ik}-x_k(\mathcal L(x))_{ij}=0
 \quad(j\ne k).
\end{equation}
Let \(\mathcal L_i(x)\) be row \(i\) of \(\mathcal L(x)\).
From \eqref{I:Cidentity}, there is a scalar \(u_i\) such that
\(\mathcal L_i(x)=u_ix\). Collecting the rows yields
\begin{equation}\label{I:Cform}
 \mathcal L(x)=ux,\quad u=(u_1,\ldots,u_s)^t.
\end{equation}

When \(s\ge2\), applying the analogous argument to first-column minors
\[
 y_iW_{kj}-y_kW_{ij}\quad(i\ne k),
\]
there is a fixed row \(v=(v_1,\ldots,v_t)\) such that
\begin{equation}\label{I:Dform}
 \mathcal K(y)=yv.
\end{equation}
Note that (\ref{I:Dform}) holds trivially when $s=1$.
It remains to find \(B\). Substituting \eqref{I:Cform} and
\eqref{I:Dform}, the mixed \(yx\) part of \(R_{i;jk}\) is
\(y_i(v_kx_j-v_jx_k)\). The mixed \(yx\) coefficients uniquely
determine a polynomial's expansion in the corner-minor basis $P_{ij}$. Therefore
\begin{equation}\label{equ:Rijk in basis}
    R_{i;jk}=-v_kP_{ij}+v_jP_{ik}.
\end{equation}
Compare the $ax$ part in (\ref{equ:row minor}) and (\ref{equ:Rijk in basis}). We find
\[
 x_j(B_{ik}+u_iv_k)-x_k(B_{ij}+u_iv_j)=0.
\]
Since $t\geq 2$, we can always find $j\neq k$. Then \(B_{ik}=-u_iv_k\), and hence \(B=-uv\), proving
\eqref{I:graphformula}.

To prove \eqref{I:lineform}, decompose each graph matrix as
\[
 \begin{pmatrix}a&x\\y&ux+yv-auv\end{pmatrix}
 =\binom1u(a,x)+\binom0{y-au}(1,v).
\]
This leads inclusion in the indicated sum. Its summands have dimensions
\(m=t+1\) and \(n=s+1\). Their intersection is \(\C \widehat{u}\widehat{v}\): a matrix
in the intersection has both its column space contained in \(\C \widehat{u}\)
and its row space contained in \(\C \widehat{v}\), and hence is a scalar multiple
of \(\widehat{u}\widehat{v}\). The sum has dimension \(m+n-1=d\), so the inclusion is
an equality.
\end{proof}

\begin{remark}
    The condition \(t\ge2\) is essential to this graph lemma. For
\(D^I_{2,2}\), one would have \(s=t=1\), and the plane
\[
 E=\left\{\begin{pmatrix}a&x\\y&a\end{pmatrix}:a,x,y\in\C\right\}
\]
has invertible cross projection and \(r(E)=1=st\), since its only
minor is the nonzero polynomial \(a^2-xy\). Nevertheless it does not
have the form \eqref{I:graphformula}. Thus, this excluded case cannot be
inferred from the same graph argument.
\end{remark}

\subsubsection{Local rigidity near the cross plane}

\begin{proposition}[Type I local rigidity]\label{I:local}
There is a neighborhood \(\mathscr U_I\) of the standard plane $E_0$
in the Grassmannian $\Gr(d,M_{n,m}(\C)))$ such that every \(E\in\mathscr U_I\) with
\(r(E)\le st\) admits unitary matrices \(L\in U(n)\) and
\(R\in U(m)\) with \(LER=E_0\). Here
\(\Gr(d,M_{n,m}(\C))\) denotes the Grassmannian of complex
\(d\)-dimensional linear subspaces of \(M_{n,m}(\C)\).
\end{proposition}

\begin{proof}
Invertibility of \(\pi|_E\) is an open condition on the Grassmannian. Choose \(\mathscr U_I\) so that this condition
holds for all its planes. Lemma~\ref{I:graph} implies
\[
 E=\widehat{u}M_{1,m}(\C)+M_{n,1}(\C)\widehat{v}
\]
for nonzero \(\widehat{u}\in M_{n,1}(\C)\) and \(\widehat{v}\in M_{1,m}(\C)\).
Choose \(L\in U(n)\), \(R\in U(m)\) with
\[
 L\widehat{u}=\|\widehat{u}\|e_1,\quad \widehat{v}R=\|\widehat{v}\|e_1^t.
\]
Since multiplication by invertible matrices permutes the full row and
column spaces, respectively, it follows that
\[
 LER=e_1M_{1,m}(\C)+M_{n,1}(\C)e_1^t=E_0.
\]
Also, the transformation \(Z\mapsto LZR\) is a target automorphism fixing
zero. This proves the asserted local
rigidity by actual unitary target automorphisms.
\end{proof}

\subsection{Type II: the local tangent rigidity}
\label{II:section}

Let $n\geq5$, let $V=\mathbb C^n$, and let $\mathcal M=\Lambda^2V$ be the
ambient space of skew-symmetric matrices, identified by their
upper-triangular entries. Set
\[
 d=2n-3,\quad d^\#=\dim \mathcal M-d=\binom{n-2}{2}.
\]
The
norm on $\mathcal M$ is $\|Z\|^2=\tfrac12\operatorname{tr}(ZZ^*)$ under the identification  
$Z\leftrightarrow\sum_{i<j}Z_{ij}\,e_i\wedge e_j\in\Lambda^2\mathbb C^n$.
For a two-plane $P\subset V$, the notation $P\w V$ means the linear
span of all $w\w v$ with $w\in P$, $v\in V$.

The quadratic identity \eqref{eq:quadratic-common}, with
$Q_2(Z)=\tfrac12Z\w Z$, puts every normalized tangent subspace in
\[
 \mathcal G=\left\{L\in\operatorname{Gr}(d,\mathcal M):
    \dim\operatorname{span}\{x\w y:x,y\in L\}\leq d^\#\right\}.
\]
The standard planes form the family
\[
 \mathcal O=\{P\w V:P\in\operatorname{Gr}(2,V)\} \subset \mathcal{G}.
\]
We will show that the locus $\mathcal G$ agrees with
this family in a neighborhood of every member of $\mathcal O$.

\begin{lemma}\label{II:linear-wedge}
Let $U$ be a complex vector space of dimension $m\geq3$. If a
complex linear map $T:U\to\Lambda^2U$ satisfies $u\w T(u)=0$
for every $u\in U$, then there is a unique $a\in U$ such that
$T(u)=u\w a$ for all $u$.
\end{lemma}

\begin{proof}
Choose a basis $f_1,\ldots,f_m$ of $U$. Since
$f_i\w T(f_i)=0$, the coefficients of $T(f_i)$ not involving
$f_i$ vanish. Thus $T(f_i)=f_i\w a_i$ for some $a_i\in U$.
Substituting $f_i+f_j$ into the hypothesis yields
\[
 f_i\w f_j\w(a_j-a_i)=0,
\]
and hence $a_j-a_i\in\operatorname{span}(f_i,f_j)$. Write $a_i=\sum_k a_{ik}f_k$. For fixed $k$, we have $a_{ik}=a_{jk}$ for any $i,j\neq k$. We then use $b_k$ to denote these common coefficients. Let $a=\sum_{k}b_kf_k$. Then $a_i-a\in\mathbb C f_i$, so
$T(f_i)=f_i\w a$ for every $i$. Linearity proves the formula.
Finally, $u\w(a-a')=0$ for every $u\in U$ forces $a=a'$, which
proves uniqueness.
\end{proof}

For $L_0=P\w V\in \mathcal O$, choose an orthogonal decomposition $V=P\oplus U$, with
$P=\operatorname{span}(e_1,e_2)$ and $m=\dim U=n-2\geq3$.
Then
\[
 \mathcal M=L_0\oplus\Lambda^2U,
 \quad L_0=\Lambda^2P\oplus(P\w U),
 \quad  \dim L_0=1+2m=d,
 \quad \dim \Lambda^2 U=d^\#.
\]

For $B\in\operatorname{Hom}(L_0,\Lambda^2U)$, define a linear map
\[
 Q_B:\operatorname{Sym}^2L_0\rightarrow\Lambda^4V
\]
by the symmetric bilinear rule
\[
 Q_B(x\odot y)=\tfrac12(x+B(x))\w(y+B(y)),
 \quad x\odot y:=\tfrac12(x\otimes y+y\otimes x).
\]
The rule is symmetric because $x+B(x)$ and $y+B(y)$ are bivectors.
Its diagonal polynomial is
\begin{equation}\label{equ:Qb=Q2}
    \widehat Q_B(x):=Q_B(x\odot x)
       =Q_2(x+B(x)).
\end{equation}

The following lemma is immediate, which locally yields a determinantal variety interpretation of $\mathcal G$.
\begin{lemma}\label{lem: reduced to rank QB}
   Write $L_B=\{x+B(x):x\in L_0\}$. Then $L_B\in \mathcal G$ if and only if $\operatorname{rank}Q_B\leq d^\#$. For $B=0$, $\operatorname{rank}Q_0=d^\#$.
\end{lemma}
\begin{proof}
Since squares $x\odot x$ span $\operatorname{Sym}^2L_0$, the first statement follows from \eqref{equ:Qb=Q2}.
Write an arbitrary $x\in L_0$ as
\[
 x=s e_1\w e_2+e_1\w u+e_2\w v,
 \quad s\in\mathbb C,\quad u,v\in U.
\]
At $B=0$,
\[
 \widehat Q_0(x)=-e_1\w e_2\w u\w v,
 \quad \operatorname{im}Q_0
      =S:=e_1\w e_2\w\Lambda^2U.
\]
So $\operatorname{rank}Q_0=\dim S=d^\#$.
\end{proof}

The $dd^\#$-dimensional Grassmannian chart centered at $L_0$ consists of graphs
$L_B=\{x+B(x):x\in L_0\}$, where
$B\in\operatorname{Hom}(L_0,\Lambda^2U)$.
In this chart, from Lemma \ref{lem: reduced to rank QB}, $\mathcal G$ is the algebraic set defined by the
$(d^\#+1)$-minors of the matrix of $Q_B$. For a tangent direction
$H\in\operatorname{Hom}(L_0,\Lambda^2U)$, let
\[
 \dot Q_H=\left.\frac{\mathrm d}{\mathrm d t}\right|_{t=0}Q_{tH}.
\]
Its diagonal polynomial is
$\dot Q_H(x\odot x)=x\w H(x)$.
We first compute $T^{Zar}_{L_0} \mathcal G$: the Zariski tangent space of $\mathcal G$ at $L_0$.

\begin{lemma}\label{lem:Zariski_tangnet_Gat0}
   $ T^{Zar}_{L_0} \mathcal{G}=\{H: \dot Q_H(\ker Q_0)\subset S\}$ and it has dimension $2m$.
\end{lemma}
\begin{proof}
Choose bases in the source and target in which
\[
 Q_0=\begin{pmatrix}J_0&0\\0&0\end{pmatrix},
 \quad \det J_0\ne0,\quad J_0\in M_{d^\#}(\mathbb C).
\]
 A minor not containing the full invertible block $J_0$ has zero first derivative, while a minor containing $J_0$ plus one extra row and column has differential equal to $\det J_0$ times the corresponding entry of $\dot Q_H$. Hence all these differentials vanish iff the lower-right block of $\dot Q_H$ vanishes, i.e.
\[
\dot Q_H(\ker Q_0)\subset S=\operatorname{im}Q_0=e_1\w e_2 \w\Lambda^2U.
\]
This exactly yields
\begin{equation}\label{II:rank-tangent}
T^{Zar}_{L_0} \mathcal{G}=\{H: \dot Q_H(\ker Q_0)\subset S\}.
\end{equation}

Let $\pi:\Lambda^4V\to\Lambda^4V/S$ be the quotient map.
Because $Q_0$ maps onto $S$, condition \eqref{II:rank-tangent}
means that $\pi\dot Q_H$ factors through $Q_0$. Equivalently,
there is a linear map $C:S\to\Lambda^4V/S$ such that, on diagonal
polynomials,
\begin{equation}\label{II:factorization}
 \pi(x\w H(x))=C(\widehat Q_0(x))\quad(x\in L_0).
\end{equation}
Write $x\in L_0$ as
\[
 x=s e_1\w e_2+e_1\w u+e_2\w v,
 \quad s\in\mathbb C,\quad u,v\in U,
\]
and then
\[
 H(x)=sH_0+T(u)+R(v),
\]
for some $H_0\in\Lambda^2U,$ and $T,R\in\operatorname{Hom}(U,\Lambda^2U)$.
Since $s e_1\w e_2\w H(x)\in S$, the left side of
\eqref{II:factorization} equals
\[
 \pi(e_1\w u\w H(x)+e_2\w v\w H(x)).
\]
Set $v=0$. Then $\widehat Q_0(x)=-e_1\w e_2\w u\w v=0$,
and the coefficient of
$s$ on the left side of \eqref{II:factorization} yields \[\pi(e_1\w u\w H_0)=0\] for every $u\in U$.
Note that $e_1\w\Lambda^3U$ has zero intersection with $S$,
and wedging by $e_1$ is injective on $\Lambda^3U$, so $u \w H_0=0$ for every $u\in U$. Since $m\geq3$, it implies $H_0=0$. The remaining part with $v=0$
yields $\pi(e_1\w u \w T(u))=0$, and similar argument implies $u\w T(u)=0$. Similarly, setting $u=0$ yields
$v\w R(v)=0$. By Lemma~\ref{II:linear-wedge}, there exist unique $a,b\in U$ such that
\begin{equation}\label{II:orbit-variation}
 H(x)=u\w a+v\w b.
\end{equation}
Conversely, such an $H$ satisfies \eqref{II:factorization}, since
\[
 \pi(x\w H(x))
 =\pi(e_1\w u\w v\w b-e_2\w u\w v\w a),
\]
which depends linearly on $u\w v$ and therefore factors through
$\widehat Q_0(x)$. Therefore, $T^{Zar}_{L_0}\mathcal G$ has dimension $2m$.
\end{proof}

\begin{lemma}\label{lem:PvwV}
Let $V$ be a complex vector space of dimension $\geq 4$. For any 2-dimensional subspace $P\subset V$,  \[P=\{v\in V:v\w V\subset P\w V\}.\]
\end{lemma}
\begin{proof}
    The inclusion from left to right is immediate. If
$v\notin P$, choose $w$ whose class in $V/P$ is linearly
independent of that of $v$; this is possible since
$\dim(V/P)=m\geq2$. Then $v\w w$ has nonzero image in
$\Lambda^2(V/P)$, whereas every element of $P\w V$ has zero
image there.
\end{proof}

\begin{proposition}[Type II local rigidity]\label{II:local-rigidity}
For every $L_0=P\w V\in\mathcal O$, there is a neighborhood
$\mathcal U$ of $L_0$ in $\operatorname{Gr}(d,\mathcal M)$ such that
\[
 \mathcal G\cap\mathcal U=\mathcal O\cap\mathcal U.
\]
\end{proposition}

\begin{proof}
The map
\[
 \Phi:\operatorname{Gr}(2,V)\rightarrow\operatorname{Gr}(d,\mathcal M),
 \quad \Phi(P)=P\w V,
\]
is holomorphic and its image is $\mathcal O$. It is injective from Lemma \ref{lem:PvwV}.
For $a,b\in U$, let $P=\operatorname{span}(e_1,e_2)\in \operatorname{Gr}(2,V)$.
For $P_t=\operatorname{span}(e_1-ta,e_2-tb)$, the corresponding tangent vector $\varphi\in \operatorname{Hom}(P,V/P)$ is
\[
\varphi(e_1)=-\bar a,\quad \varphi(e_2)=-\bar b,
\]
where $\bar a,\bar b$ denote the images in $V/P$.
For $\Phi(P)=P\wedge V$, the differential at $P$ yields
$
d\Phi_{P}(\varphi)\in \operatorname{Hom}(P\wedge V,\mathcal M/(P\wedge V))
$, and 
\[
d\Phi_{P}(\varphi)(p\wedge v)=\varphi(p)\wedge v \bmod P\wedge V.
\]
Evaluating on
\[
x=s\,e_1\wedge e_2+e_1\wedge u+e_2\wedge v,
\]
we have
\[
d\Phi_{P}(\varphi)(x)
=
-a\wedge u-b\wedge v
=
u\wedge a+v\wedge b
\bmod P\wedge V.
\]
Suppose $d\Phi_{P}(\varphi)=0$, setting $v=0$ yields $u\w a \in P\w V$ for every $u\in V$. From Lemma \ref{lem:PvwV}, we have $a\in P$. Similarly, setting $u=0$ yields $b\in P$.
Thus $\varphi=0$ and this proves that $\mathrm{d}\Phi_{P}$ is injective and that its image is exactly
\eqref{II:orbit-variation}. The Grassmannian is compact, so the
injective holomorphic immersion $\Phi$ is an embedding. Thus
$\mathcal O$ is a smooth complex submanifold of dimension $2m$.  
Moreover, at $L_0=P\w V$, we have \[T_{L_0}\mathcal O=T^{Zar}_{L_0} \mathcal{G}.\]

As $T^{Zar}_{L_0}\mathcal G$ has dimension $2m$, we may therefore select $dd^\#-2m$ of the rank-defining minors whose differentials are
linearly independent at $B=0$. By the implicit function theorem their common
zero set is, locally, a smooth complex manifold $Z$ of dimension $2m$. This auxiliary $Z$
contains $\mathcal G$, which contains $\mathcal O$. The inclusion $\mathcal O\to Z$ has invertible differential at $L_0$,
so the inverse function theorem identifies $\mathcal O$ with a neighborhood of $L_0$ in $Z$. After
shrinking the Grassmannian neighborhood, $Z=\mathcal O$ there. The inclusions $\mathcal O\subset \mathcal G\subset Z$
now yield the required local equality.
\end{proof}
\begin{remark}
    The condition $n \geq 5$ is essential. For $D_4^{II}$, $ \mathcal G=\{L\in\operatorname{Gr}(5,\Lambda^2\mathbb C^4):
    \dim\operatorname{span}\{x\w y:x,y\in L\}\leq 1\}\cong \mathbb P^5$, as the inequality is automatically satisfied. While $\mathcal O$ is 4-dimensional.

\end{remark}
\subsection{Type III: eliminating the remaining tangent-plane deformations}
\label{III:sec:local-rigidity}

Throughout this subsection, $n\geq3$, $d=n$, and $m=n-1\geq2$. We use the
ordinary compound matrices $C_k(X)=\bigwedge^kX \in \operatorname{Hom}(\Lambda^k \mathbb{C}^n, \Lambda^k \mathbb{C}^n)$ and their full
Hilbert--Schmidt norms. In this case, $Q_k=C_k$. The quadratic and cubic
identities therefore apply directly.

The proof strategy for Proposition \ref{prop:exact-tangent} in this section is as follows. For each plane $E$, writing the vector of minors in a quadratic monomial basis yields a coefficient matrix. Note that $\dim\Span C_2(E)$ is exactly the rank of the coefficient matrix, which is a lower semicontinuous function with respect to the plane.
Therefore, if $E$ is a plane in a sufficiently small neighborhood of the standard plane $E_0$,
$\dim\Span C_2(E)\geq \dim\Span C_2(E_0)=d^\#$. Then the quadratic identity in Lemma \ref{lem:quadratic-common} forces $\dim\Span C_2(E)= \dim\Span C_2(E_0)$. In Section \ref{subsec:normalformremaining}, we prove 
Lemma \ref{III:eq:hankel} and it puts every nearby tangent subspace in the form $\cH_b$ (cf. \eqref{III:eq:hankel}), parametrized by a small commutative associative
multiplication $b$. It does not yet force $b=0$, and the cubic identity is needed. In Section \ref{subsec:deformcalcu}, we do some calculation and obtain the expansion of the diagonal defect of \eqref{eq:cubic-common}(cf. Proposition \ref{III:prop:exact-Q2}) in the deformation. In Section \ref{subsec:fourthgap}, we prove that the vanishing of the fourth order term in the expansion will force $b=0$. Then we finish the proof in Section \ref{sec:type-III-local-rigidity-done}.

\subsubsection{The normal form with remaining deformation}
\label{subsec:normalformremaining}
Write \(n=m+1\), and let
\begin{equation}
E_0=\left\{
\begin{pmatrix}
u& \frac{x^t}{\sqrt2}\\
\frac{x}{\sqrt2}&0
\end{pmatrix}:u\in\C,\ x\in\C^m
\right\}.
\label{III:eq:E0}
\end{equation}
Here $\Gr(d,\Sym_n\C)$ denotes the Grassmannian of $d$-dimensional
complex subspaces of the symmetric-matrix space. We use its usual
topology, equivalently convergence of orthogonal projections. In this
subsection the distinguished first row and column are indexed by $0$;
the lower rows and columns have indices $1,\ldots,m$. Thus $E_{00}$ is
the matrix with a single $1$ in the distinguished position. The
factor $1/\sqrt2$ in the displayed parametrization makes its source
coordinates orthonormal.

We define the notion of linear reduction here for the commutative algebra argument in Lemma \ref{III:lem:hankel}.
\begin{definition}
\label{def:linearreduction}
Let $D(z_0,z)$ be an $m\times m$ symmetric matrix-valued linear function. Set $\mathcal D_{ij}(x)=D_{ij}(1,x)$.
On polynomials of degree at most two, the linear reduction $\mathrm{NF}$ is defined by
\[
\mathrm{NF}(x_i x_j)=\mathcal D_{ij}(x),
\quad
\mathrm{NF}(1)=1,
\quad
\mathrm{NF}(x_k)=x_k,
\]
extended linearly.
\end{definition}

The following lemma shows that the plane $E$ near $E_0$ for which $\dim\operatorname{im}Q_2(E)=d^\#$ has a possible remaining deformation from the standard plane.
\begin{lemma}
\label{III:lem:hankel}
There is a neighborhood \(\mathscr U\) of \(E_0\) in
\(\Gr(d,\Sym_n\C)\) with the following property.
If \(E\in\mathscr U\) and
\begin{equation}\label{equ:IIIsharp plane}
    \dim\Span C_2(E)=\binom n2=d^\#,
\end{equation}

then:
\begin{enumerate}
\item the scheme
\[
Z_E:=\{[X]\in\mathbb P(E):\operatorname{rank}X\le1\}
\]
has length \(d\) and spans $\mathbb P(E)$ scheme-theoretically;
\item after choosing a support point of \(Z_E\), a unitary congruence,
and a basis \(1,e_1,\ldots,e_m\), the plane is of the form
\begin{equation}
\cH_b=
\left\{
J_b(u,\xi)=
\begin{pmatrix}
u&\xi^t\\
\xi&B_\xi
\end{pmatrix}:u\in\C,\ \xi\in(\C^m)^*
\right\},
\label{III:eq:hankel}
\end{equation}
where
\[
(B_\xi)_{ij}=\xi(b(e_i,e_j))
\]
for a commutative associative multiplication
\[
b:\Sym^2\C^m\rightarrow\C^m.
\]
As \(E\to E_0\), one can make these choices so that \(b\to0\).
\end{enumerate}
The scheme $Z_E$ is defined by all restricted $2\times2$ minors.
Its length is the dimension of its coordinate algebra after choosing
an affine chart containing it. To say that it spans $\mathbb P(E)$ scheme-theoretically
means that no nonzero linear form vanishes on the scheme; its set of
support points alone need not span.
\end{lemma}

\begin{proof}
For a plane sufficiently close to $E_0$, the entries in its first row
are independent linear coordinates. In these coordinates it has the form
\[
A(z_0,z)=\begin{pmatrix}z_0&z^t\\z&D(z_0,z)\end{pmatrix},
\quad z\in\C^m,
\]
where $D$ is symmetric, linear, and tends to zero as $E\to E_0$.
The minors
\[
\mathcal{Q}_{ij}=z_0D_{ij}(z_0,z)-z_i z_j\quad(1\le i\le j\le m)
\]
are linearly independent, as $z_0D_{ij}(z_0,z)$ has no
quadratic part with respect to $z$. The equality
of dimensions in \eqref{equ:IIIsharp plane} implies that every restricted
$2\times2$ minor is a constant linear combination of the $\mathcal{Q}_{ij}$.

Set $\mathcal{D}_{ij}(x)=D_{ij}(1,x)$ and $f_{ij}=x_i x_j-\mathcal{D}_{ij}(x)$. Each cross minor
$x_k\mathcal{D}_{ij}-x_i \mathcal{D}_{kj}$ is a constant linear combination of the $f_{ab}$;
therefore its linear reduction (cf. Definition \ref{def:linearreduction}) is zero. This follows because the homogeneous cross minor
$z_kD_{ij}-z_iD_{kj}$ belongs to the span of the $\mathcal{Q}_{ab}$, and setting
$z_0=1$ yields this assertion because $\mathcal{Q}_{ab}(1,x)=-f_{ab}$. 

Next we prove that on the affine chart $z_0=1$, the coordinate algebra of $Z_E$ \[
\mathcal A:=\C[x_1,\ldots,x_m]/(f_{ij}),
\]
has dimension $m+1=d$. The following argument is precisely the mechanism of the commuting multiplication matrices criterion for border bases, applied to the order ideal $\{1,x_1,\ldots,x_m\}$ (cf. \cite{KKR05}).
For an affine-linear
polynomial $h$, define $M_kh=\mathrm{NF}(x_kh)$.
On the vector space $V=\C\{1,x_1,\ldots,x_m\}$, we have
\[
M_i1=x_i,\quad M_ix_j=\mathcal{D}_{ij}(x).
\]
The preceding reduced cross-minor identity says
\[(M_kM_i-M_iM_k)x_j=0.\] The same commutator vanishes on $1$ by
the symmetry of $D$, so the $M_i$ commute. Let $M=(M_1,\cdots, M_m)$. Note that \[f_{ij}(M)1=M_iM_j1-\mathcal{D}_{ij}(M)1=0,\]
as $\mathcal{D}_{ij}(x)$ is linear. For every polynomial $\varphi\in \mathbb{C}[x_1,\cdots,x_m]$, by the commutativity of the $M_i$, $f_{ij}(M)$ also annihilates every $\varphi(M)1$.
Consequently $\varphi\mapsto \varphi(M_1,\ldots,M_m)1 \in V$ factors through
\[
\mathcal A=\C[x_1,\ldots,x_m]/(f_{ij}).
\]
As $f_{ij}=x_i x_j-\mathcal{D}_{ij}(x)$, repeated replacement of quadratic factors lowers total degree, so
every class in $\mathcal A$ lies in
\[
\Span\{1,x_1,\ldots,x_m\}.
\]
Under the map $\varphi\mapsto \varphi(M)1$, their images are the prescribed
basis $1,x_1,\ldots,x_m$ of $V$, so they are independent in
$\mathcal A$. Hence they are a basis of $\mathcal A$, which has
dimension $m+1=d$.

At $E_0$, we know coordinte algebra of $Z_{E_0}$ on $z_0=1$ is 
\[\mathcal{A}_{E_0}=\mathbb{C}[x_1,\cdots,x_m]/(x_ix_j)_{1\leq i,j\leq m},\]
and the set of support points $\supp(Z_{E_0})=\{[E_{00}]\}$. 
It has no intersection with $\{z_0=0\}$. The same remains true for all sufficiently close planes,
by compactness of that hyperplane. All restricted minors are generated
by the $\mathcal{Q}_{ij}$, so their entire projective zero scheme lies in the
chart $z_0\ne0$ and is precisely $\operatorname{Spec}\mathcal A$, the finite affine scheme whose coordinate algebra is $\mathcal A$.
It has length $d$ and spans $\mathbb P(E)$, since no nonzero linear
form vanishes in $\mathcal A$. This proves the first assertion.

Every support point of this scheme approaches $[E_{00}]$ as
$E\to E_0$, because that is the only support point of the limiting
scheme. Choose one such point, represented by a rank-one symmetric
matrix $vv^t$. A unitary congruence approaching the identity takes
its projective class to $[E_{00}]$. The transformed plane still lies
in the first-row coordinate chart and contains $E_{00}$. Its lower
block $D(z_0,z)$ therefore has no $z_0$ term; write it as $B_\xi$. As $f_{ij}=x_ix_j-\mathcal{D}_{ij}(x)$ and $\mathcal{D}_{ij}$ is linear, we know in this chart
\[
x_i x_j=\sum_k b^k_{ij}x_k
\]
with zero constant terms (as $D(z_0,z)$ has no $z_0$ term). They define the commutative associative
multiplication $b$ on the ideal $\Span\{x_1,\ldots,x_m\}$, and the quotient algebra multiplication restricts to that ideal because
the relations have zero constant term. Its commutativity and
associativity are inherited from the polynomial quotient. The matrix
of the plane is exactly \eqref{III:eq:hankel}. All the operations
used on the target were unitary congruences. Since the lower-block
coefficients tend to zero, $b\to0$ as $E\to E_0$.
\end{proof}

\subsubsection{Some deformation calculation}
\label{subsec:deformcalcu}
We first specify the finite-dimensional spaces, coordinates and norms used in the deformation
calculation. All adjoints and orthogonal projections below refer to
the stated Hermitian metrics.

Let $W=\mathbb C^m$, $m\ge2$, and let
\[
c:\operatorname{Sym}^2W\to W
\]
be a commutative associative multiplication. In an orthonormal basis
$(e_1,\dots,e_m)$, write
\[
c(e_i,e_j)=\sum_{k=1}^m c_{ij}^{\,k}e_k.
\]
For $\xi\in W^*$, let $\xi_k=\xi(e_k)$ and define
\[
(C_\xi)_{ij}=\sum_{k=1}^m c_{ij}^{\,k}\xi_k.
\]
Equivalently, $C_\xi$ is the symmetric matrix of the bilinear form
\[
(u,v)\mapsto \xi(c(u,v)),\quad u,v\in W.
\]
This yields a linear map
\[
C:W^*\to\mathscr S,\quad \xi\mapsto C_\xi,
\]
where $\mathscr S=\operatorname{Sym}^2W^*$ is identified with the space of
symmetric $m\times m$ matrices. Both $W^*$ and $\mathscr S$ are equipped
with the Hermitian metrics induced by the chosen orthonormal basis. In particular, the matrix metric is the full
Hilbert--Schmidt metric. A unitary change of basis acts
contragrediently on covector columns and preserves all norms below.
In particular,
\[
\langle H,G\rangle_{\mathscr S}=\sum_{i,j=1}^mH_{ij}\overline{G_{ij}},
\quad
\|H\|_{\mathscr S}^2=\sum_i|H_{ii}|^2+2\sum_{i<j}|H_{ij}|^2.
\]
 The adjoint $C^*:\mathscr S\to W^*$ is characterized by
$\langle C_\xi,H\rangle=\langle \xi,C^*H\rangle$ and has coordinates
\[
(C^*H)_k=\sum_{i,j=1}^m\overline{c_{ij}^{\,k}}H_{ij}.
\]
Define also
\[
\mathscr T=W^*\otimes\Lambda^2W^*,\quad
\mathscr V=\Sym^2(\Lambda^2W^*),
\]
represented respectively by $m\times\binom m2$ matrices and symmetric
$\binom m2\times\binom m2$ matrices. Their norms are the full
Hilbert--Schmidt norms in the dual orthonormal and increasing wedge
bases. For
\(H\in\mathscr S\), set
\begin{align}
(D_cH)_{i,jk}
 &=\sum_r(c_{ik}^{\,r}H_{jr}-c_{ij}^{\,r}H_{kr}),\\
\mathsf S_\xi(H)_{i,jk}
 &=\xi_jH_{ik}-\xi_kH_{ij},\\
\mathsf K(C,H)_{ij,k\ell}
 &=C_{ik}H_{j\ell}+H_{ik}C_{j\ell}
   -C_{i\ell}H_{jk}-H_{i\ell}C_{jk}.
\end{align}
$D_c:\mathscr S\to\mathscr T$ is linear,
$\mathsf S_\xi:\mathscr S\to\mathscr T$ is linear for fixed $\xi$, and
$\mathsf K:\mathscr S\times\mathscr S\to\mathscr V$ is bilinear.
The map $E_c:\mathscr S\to\mathscr V$ is the unique linear map with
\begin{equation}
    E_c(\xi \xi^t)=C_2(C_\xi).
\end{equation} 
The letter $K$ without a sans-serif font below denotes a skew-Hermitian
operator on $\mathscr S$, and is distinct from the bilinear map
$\mathsf K$.

For a linear subspace $L$ of a Hermitian space, $P_L$ denotes its
orthogonal projection. In this subsection,
\[\widetilde Q_2(X,Y):=C_2(X+Y)-C_2(X)-C_2(Y)\] is the polarized second compound
on symmetric $(m+1)\times(m+1)$ matrices. Its lower-block analogue
on $m\times m$ matrices is $\mathsf K$. 
Recall that 
\[
\cH_b=
\left\{
J_b(u,\xi)=
\begin{pmatrix}
u&\xi^t\\
\xi&B_\xi
\end{pmatrix}:u\in\C,\ \xi\in(\C^m)^*
\right\}.
\]
We set
\[
R_b=\Span C_2(\cH_b),\quad N_b=\cH_b^\perp.
\]
Let $T_b:R_b\to N_b$ be the isometry furnished by the quadratic
identity (cf. Lemma \ref{lem:quadratic-common}). Note that $d^\#=\dim \mathscr S=\dim R_b=\dim N_b$. To describe the coordinate matrix for $T_b$, we use $\mathscr S$ and the following graph frame. For the convenience of further computation, we write $b=\varepsilon c$ for $\varepsilon>0$.

For the block matrix expression of the second compound matrix, we split the index sets into $(0,i)$, $1\le i\le m$, followed by $(i,j)$, $1\le i<j\le m$. This is the decomposition corresponding to a distinguished first coordinate and the $m$ lower coordinates.
\begin{lemma}\label{lem:coordinate_rho_nu}
 For the commutative associative multiplication $c$, $ \rho_{\varepsilon c} :\mathscr S \rightarrow R_b, \nu_{\varepsilon c}:\mathscr S\rightarrow N_b$ given as follows are well-defined, and are linear isomorphisms. 
\begin{equation}
\rho_{\varepsilon c}(H)
=\begin{pmatrix}
-H&\varepsilon D_cH\\
\varepsilon(D_cH)^t&\varepsilon^2E_cH 
\end{pmatrix}
\label{equ:rho-eps_c}
\end{equation}

\begin{equation}
    \nu_{\varepsilon c}(H)
=\begin{pmatrix}
0&-\frac{\varepsilon}{2}(C^*H)^t\\
-\frac{\varepsilon}{2}C^*H&H
\end{pmatrix}
\label{equ:nuspsilonc}
\end{equation}
\end{lemma}
\begin{proof}
The formula for \(E_c\) is obtained by polarizing
\(E_c(\xi \xi^t)=C_2(C_\xi)\); explicitly,
\[
(E_cH)_{ij,k\ell}
=\sum_{a,b}H_{ab}
\bigl(c_{ik}^{\,a}c_{j\ell}^{\,b}
      -c_{i\ell}^{\,a}c_{jk}^{\,b}\bigr)
\quad(i<j,\ k<\ell).
\]
To see that the image of $\rho_{\varepsilon c}$ lies in $R_b$, let $(C^{\,a})_{ij}=c_{ij}^{\,a}$.
For $i,j,k$, the entry
\[
(D_cC^{\,a})_{i,jk}
=\sum_r\bigl(c_{ik}^{\,r}c_{jr}^{\,a}
             -c_{ij}^{\,r}c_{kr}^{\,a}\bigr)
\]
is the coefficient of $e_a$ in
$c(e_j,c(e_i,e_k))-c(e_k,c(e_i,e_j))$, hence is zero by
commutativity and associativity. Similarly,
\[
(E_cC^{\,a})_{ij,k\ell}
=\sum_{r,s}c_{rs}^{\,a}
 (c_{ik}^{\,r}c_{j\ell}^{\,s}-c_{i\ell}^{\,r}c_{jk}^{\,s})
\]
is the coefficient of $e_a$ in
\[
c(c(e_i,e_k),c(e_j,e_\ell))
-c(c(e_i,e_\ell),c(e_j,e_k)),
\]
which is also zero. Thus $D_cC_\xi=E_cC_\xi=0$ for every $\xi$.
Expanding the minors of $J_{\varepsilon c}(u,\xi)$ now yields
\begin{equation}\label{equ: C2action on hankel}
C_2(J_{\varepsilon c}(u,\xi))
=\begin{pmatrix}
\varepsilon uC_\xi-\xi \xi^t&\varepsilon D_c(\xi \xi^t)\\
\varepsilon(D_c(\xi \xi^t))^t&\varepsilon^2E_c(\xi \xi^t)
\end{pmatrix}\\
=\rho_{\varepsilon c}(\xi \xi^t-\varepsilon uC_\xi). 
\end{equation}
The rank-one symmetric matrices $\xi \xi^t$ span $\mathscr S$, so the
values at $u=0$ already span $\rho_{\varepsilon c}(\mathscr S)$.
Conversely the formula puts every second compound in that image.
Hence this image is exactly $R_{\varepsilon c}$, and $\rho_{\varepsilon c}$ is
injective because its upper-left block is $-H$. Hence $\rho_{\varepsilon c}$ is an isomorphism, as $\dim \mathscr S=\dim R_b=d^\#.$

For $\nu_{\varepsilon c}$, consider the pair of a symmetric block matrix
$X=\left(\begin{smallmatrix}\alpha&v^t\\v&H\end{smallmatrix}\right)$
with $J_{\varepsilon c}(u,\xi)$. Orthogonality for all $(u,\xi)$ means
\[
\alpha\bar u+2\langle v,\xi\rangle+
\varepsilon\langle H,C_\xi\rangle=0,
\]
so $\alpha=0$ and $2v+\varepsilon C^*H=0$. This proves that
$\nu_{\varepsilon c}$ is an isomorphism onto $N_{\varepsilon c}$.
\end{proof}

The graph frames $\rho_b:\mathscr S\to R_b$ and
$\nu_b:\mathscr S\to N_b$ displayed above yield its coordinate matrix
\[
\widetilde T_b=\nu_b^{-1}T_b\rho_b.
\]
Note that $\rho_b,\nu_b$ are not isometric unless $b=0$. Write the Gram operators as
$G_R=\rho_b^*\rho_b$ and $G_N=\nu_b^*\nu_b$. Since \(T_b\) is an isometry,
\[
\langle T_bx,T_by\rangle_{N_b}=\langle x,y\rangle_{R_b}.
\]
Writing \(x=\rho_b(\xi)\), \(y=\rho_b(\eta)\), and using
\(T_b\rho_b(\xi)=\nu_b(\widetilde T_b\xi)\), we get
\[
\langle G_N\widetilde T_b\xi,\widetilde T_b\eta\rangle_{\mathscr S}
=
\langle G_R\xi,\eta\rangle_{\mathscr S}.
\]
Hence the coordinate isometric equation is 
\[
\widetilde T_b^*G_N\widetilde T_b=G_R.
\]

We then compute the expansions in the following
propositions.

\begin{proposition}
\label{III:prop:exact-Q1}
For $b=\varepsilon c$ with $\varepsilon>0$, write
$\widetilde T_\varepsilon=\widetilde T_{\varepsilon c}$ and suppose
this coordinate isometry satisfies
\begin{equation}\label{equ:expand-T-second-order}
    \widetilde T_\varepsilon=I+\varepsilon^2L+o(\varepsilon^2),
\quad \widetilde T_\varepsilon^*G_N\widetilde T_\varepsilon=G_R.
\end{equation}
Then
\begin{equation}
L=D_c^*D_c-\frac14CC^*+K,\quad K^*=-K,
\label{III:eq:L-expansion}
\end{equation}
for some skew-Hermitian operator $K$ on $\mathscr S$.

\end{proposition}
\begin{proof}
    From Lemma \ref{lem:coordinate_rho_nu},
taking squared Hilbert–Schmidt norms of the blocks yields the exact Gram operators $G_N=\nu_{\varepsilon c}^*\nu_{\varepsilon c}, G_R=\rho^*_{\varepsilon c}\rho_{\varepsilon c}$ as
\begin{align}
G_N&=I+\frac{\varepsilon^2}{2}CC^*,\label{equ:GN}\\
G_R&=I+2\varepsilon^2D_c^*D_c+\varepsilon^4E_c^*E_c \label{equ:GR}.
\end{align}
As the coordinate map satisfies
$\widetilde T_\varepsilon^*G_N\widetilde T_\varepsilon=G_R$, comparing the coefficient of $\varepsilon^2$ in this equality
yields \begin{equation}
L+L^*+\frac12 CC^*=2D_c^*D_c.
\label{III:eq:L-selfadjoint}
\end{equation}
Let $K=\frac{L-L^*}{2}$, we obtain \eqref{III:eq:L-expansion}. 
\end{proof}

\begin{proposition} \label{III:prop:exact-Q2}
We inherit the notation from Proposition \ref{III:prop:exact-Q1} and let
\[
P=\xi \xi^t,\quad
\mathsf R_c=D_c^*D_c+\frac14CC^*,\quad
H_\xi=(\mathsf R_c+K)P .
\]
At $u=0$ write $A_\varepsilon(\xi)=J_{\varepsilon c}(0,\xi)$. 
Then $F_2(\xi)=\nu_{\varepsilon c}(\widetilde T_\varepsilon P)$. Suppose the coordinate isometry satisfies
\begin{equation}\label{equ:expand-T-second-order-rep}
    \widetilde T_\varepsilon=I+\varepsilon^2L+o(\varepsilon^2),
\quad \widetilde T_\varepsilon^*G_N\widetilde T_\varepsilon=G_R.
\end{equation}
The diagonal
defect of \eqref{eq:cubic-common} is the following real-valued
polynomial in $\xi$ and $\bar \xi$:
\[
\Delta_\varepsilon(\xi)=
\bigl\|P_{R_{\varepsilon c}^{\perp}}
 \widetilde Q_2(A_\varepsilon(\xi),F_2(\xi))\bigr\|^2
-\|C_3(A_\varepsilon(\xi))\|^2.
\]
For an admissible isometry this defect vanishes. Its expansion is
$\Delta_\varepsilon(\xi)=\varepsilon^4Q_{c,K}(\xi)+o(\varepsilon^4)$,
where
\begin{equation}
\;
Q_{c,K}(\xi)=
2\|\mathsf S_\xi(H_\xi)\|_{\HS}^2
+2\operatorname{Re}
\langle\mathsf K(C_\xi,P),\mathsf K(C_\xi,H_\xi)\rangle_{\HS}
-2\|Z_2(\xi,C_\xi)\|_{\HS}^2 .
\;
\label{III:eq:exact-Q}
\end{equation}
Here, for $i<j$ and $k<\ell<s$,
\begin{align}
(Z_2)_{ij,k\ell s}
={}&\xi_k\bigl((C_\xi)_{i\ell}(C_\xi)_{js}
             -(C_\xi)_{is}(C_\xi)_{j\ell}\bigr)\notag\\
&-\xi_\ell\bigl((C_\xi)_{ik}(C_\xi)_{js}
             -(C_\xi)_{is}(C_\xi)_{jk}\bigr)\notag\\
&+\xi_s\bigl((C_\xi)_{ik}(C_\xi)_{j\ell}
             -(C_\xi)_{i\ell}(C_\xi)_{jk}\bigr).
\label{III:eq:Z2}
\end{align}
\end{proposition}

\begin{proof}
Write
$\widetilde T_\varepsilon=\widetilde T_{\varepsilon c}, T_\varepsilon=T_{\varepsilon c}$. At $u=0$, let
\[
A_\varepsilon(\xi)=
\begin{pmatrix}0&\xi^t\\\xi&\varepsilon C_\xi\end{pmatrix}=J_{\varepsilon c}(0,\xi),
\quad \chi=C^*P,\quad M=LP,\quad
A_1=\frac12(\xi \chi^t+\chi \xi^t).
\]
From \eqref{equ:GR}, we have
\[
G^{-1}_R=(\rho_{\varepsilon c}^*\rho_{\varepsilon c})^{-1}
=I-2\varepsilon^2D_c^*D_c+O(\varepsilon^4).
\]
Let \[V_\varepsilon=P_{R_{\varepsilon c}^{\perp}}
\widetilde Q_2(A_\varepsilon(\xi),F_2(\xi)).\] 
From \eqref{eq:quadratic-common}, together with \eqref{equ: C2action on hankel}, we know
 \[ F_2(\xi)=T_{\varepsilon}C_2(A_\varepsilon(\xi))=T_\varepsilon\rho_{\varepsilon c}(P)=\nu_{\varepsilon c}(\widetilde T_\varepsilon P).\]
Here is the block calculation. Set
\[H_\varepsilon=\widetilde T_\varepsilon P,\quad
\quad \chi_\varepsilon=C^*H_\varepsilon, \quad
J_\varepsilon=(\xi \chi_\varepsilon^t+\chi_\varepsilon \xi^t)/2.\]
Write \[\Omega_\varepsilon=\widetilde Q_2(A_\varepsilon(\xi
),F_2(\xi))=C_2(A_\varepsilon(\xi)+F_2(\xi))-C_2(A_\varepsilon(\xi))-C_2(F_2(\xi)).\] 
By a straightforward computation, we expand each
$2\times2$ minor and present the exact blocks
\begin{align*}
\Omega_\varepsilon&=
\begin{pmatrix}
\varepsilon J_\varepsilon&Y_\varepsilon\\
Y_\varepsilon^t&Z_\varepsilon
\end{pmatrix},\\
Y_\varepsilon&=\mathsf S_\xi(H_\varepsilon)
             -\frac{\varepsilon^2}{2}\mathsf R(\chi_\varepsilon,C_\xi),\\
Z_\varepsilon&=\varepsilon\mathsf K(C_\xi,H_\varepsilon).
\end{align*}
where $\mathsf R(\chi,G)_{i,jk}=\chi_jG_{ik}-\chi_kG_{ij}$.
Its orthogonal projection onto $R_{\varepsilon c}$ is $\rho_{\varepsilon c}(h_\varepsilon)$, where
\[
h_\varepsilon=G_R^{-1}
 \bigl(-\varepsilon J_\varepsilon
       +2\varepsilon D_c^*Y_\varepsilon
       +\varepsilon^2E_c^*Z_\varepsilon\bigr).
\]
To verify this formula, use the defining property of orthogonal projection:
for every $H\in\mathscr S$,
\begin{equation}
    \langle P_{R_{\varepsilon c}}\Omega_\varepsilon,\rho_{\varepsilon c}(H)\rangle
=
\langle \Omega_\varepsilon,\rho_{\varepsilon c}(H)\rangle. \label{equ:projectiontoRepsilon}
\end{equation}
The left-hand side of \eqref{equ:projectiontoRepsilon} equals
\[
\langle \rho_{\varepsilon c}(h_\varepsilon),\rho_{\varepsilon c}(H)\rangle
=
\langle G_Rh_\varepsilon,H\rangle,
\]
while from \eqref{equ:rho-eps_c}, the right-hand side of \eqref{equ:projectiontoRepsilon} is
\[
\left\langle
-\varepsilon J_\varepsilon
+2\varepsilon D_c^*Y_\varepsilon
+\varepsilon^2E_c^*Z_\varepsilon,
H
\right\rangle.
\]
Hence
\[
G_Rh_\varepsilon
=
-\varepsilon J_\varepsilon
+2\varepsilon D_c^*Y_\varepsilon
+\varepsilon^2E_c^*Z_\varepsilon.
\]
By the assumption \eqref{equ:expand-T-second-order-rep}, \[H_\varepsilon=\widetilde T_\varepsilon P=P+\varepsilon^2M+o(\varepsilon^2).\] Also,
\[\mathsf S_\xi(P)_{i,jk}
 =\xi_jP_{ik}-\xi_kP_{ij}=0,\] as $P=\xi\xi^t$. Hence $\mathsf S_\xi(P)=0$, and we have
$Y_\varepsilon=O(\varepsilon^2)$,
$Z_\varepsilon=O(\varepsilon)$. Also from Proposition \ref{III:prop:exact-Q1}, we have $\chi_\varepsilon=\chi+O(\varepsilon^2)$, and then $J_\varepsilon=A_1+O(\varepsilon^2)$. Thus
$h_\varepsilon=-\varepsilon A_1+O(\varepsilon^3)$.
Subtracting $\rho_{\varepsilon c}(h_\varepsilon)$ from
$\Omega_\varepsilon$ therefore yields
\begin{align*}
(V_\varepsilon)_{00}&=O(\varepsilon^3),\\
(V_\varepsilon)_{01}&=\varepsilon^2\Theta+o(\varepsilon^2),\\
(V_\varepsilon)_{11}&=\varepsilon\mathsf K(C_\xi,P)
 +\varepsilon^3\bigl(\mathsf K(C_\xi,M)+E_cA_1\bigr)
 +o(\varepsilon^3),
\end{align*}
where
\[
\Theta=\mathsf S_\xi(M)-\frac12\mathsf R(\chi,C_\xi)+D_cA_1,\quad
\mathsf R(\chi,C)_{i,jk}=\chi_jC_{ik}-\chi_kC_{ij}.
\]
By inserting
$A_1=\frac{1}{2}(\xi \chi^t+\chi \xi^t)$ into the definition of $D_c$ and collecting
the terms with $\chi_j,\chi_k$ and those with $\xi_j,\xi_k$, we have
\begin{equation}
    D_cA_1=\frac12\mathsf S_\xi(C_\chi)
       +\frac12\mathsf R(\chi,C_\xi). \label{equ:polarization-D-c}
\end{equation}

Substituting this into the definition of \(\Theta\), the
\(\mathsf R(\chi,C_\xi)\)-terms cancel, and we obtain
\[
\Theta
=\mathsf S_\xi(M)+\frac12\mathsf S_\xi(C_\chi)
=\mathsf S_\xi\left(M+\frac12C_\chi\right).
\]
From \eqref{III:eq:L-expansion}, we know
\[
M=LP=\left(D_c^*D_c-\frac14CC^*+K\right)P.
\]
Also \(\chi=C^*P\), hence \(C_\chi=CC^*P\). Therefore
\[
M+\frac12C_\chi
=
\left(D_c^*D_c+\frac14CC^*+K\right)P
=
H_\xi.
\]
Thus
\[
\Theta=\mathsf S_\xi(H_\xi).
\]
By polarization of $E_c(\xi\xi^t)=C_2(C_\xi)$, we have
\begin{equation}
    E_cA_1=\frac12\mathsf K(C_\xi,C_\chi), \label{equ:polarization-E-c}
\end{equation}
then we get
\[
\mathsf K(C_\xi,M)+E_cA_1
=
\mathsf K(C_\xi,M)+\frac12\mathsf K(C_\xi,C_\chi)
=
\mathsf K\left(C_\xi,M+\frac12C_\chi\right)
=
\mathsf K(C_\xi,H_\xi).
\]
In summary, we have
\begin{align*}
(V_\varepsilon)_{00}&=O(\varepsilon^3),\\
(V_\varepsilon)_{01}&=\varepsilon^2\mathsf S_\xi(H_\xi)+o(\varepsilon^2),\\
(V_\varepsilon)_{11}
&=
\varepsilon\mathsf K(C_\xi,P)
+\varepsilon^3\mathsf K(C_\xi,H_\xi)
+o(\varepsilon^3).
\end{align*}
Finally, to present the block matrix expression, we split the third-compound indices into $(0,i,j)$ followed
by $(i,j,k)$, with all lower indices increasing. Expansion of a
$3\times3$ determinant along its distinguished row yields
\begin{equation}\label{equ:C_3_A}
C_3(A_\varepsilon(\xi))=
\begin{pmatrix}
-\varepsilon\mathsf K(C_\xi,P)&\varepsilon^2Z_2(\xi,C_\xi)\\
\varepsilon^2Z_2^t(\xi, C_\xi)&\varepsilon^3C_3(C_\xi)
\end{pmatrix}.
\end{equation}
Then the order-$\varepsilon^2$ squared norms cancel. Comparison at order
$\varepsilon^4$ yields \eqref{III:eq:exact-Q}. 
\end{proof}

In Proposition \ref{III:prop:metric-gap}, we will prove the expansion in \eqref{equ:expand-T-second-order}(or \eqref{equ:expand-T-second-order-rep}) holds before this calculation is applied.

\begin{remark}\label{expansion_holds_for_sequence}
    Note that the remainders are uniform when $c,L,\xi$ range over bounded sets, as the graph frames $\rho_{\epsilon c}, \nu_{\epsilon c}$ are polynomials and $G^{-1}_{R}$ is uniformly bounded.
\end{remark}

\subsubsection{The fourth order gap rigidity}\label{subsec:fourthgap}
We inherit the notation in Section \ref{subsec:deformcalcu}, unless otherwise specified. The purpose of this section is to prove the following proposition.
\begin{proposition}\label{III:lem:algebra-obstruction}
Let $W=\mathbb{C}^m,m\geq 2$.  If $c:\Sym^2W\to W$ is commutative and associative,
and $Q_{c,K}(\xi)=0$ for every $\xi\in W^*$ and some $K^*=-K$, then
$c=0$.
\end{proposition}
In this section $R=\mathsf R_c=D_c^*D_c+\tfrac14CC^*$ is an operator on
$\mathscr S$, not the compound range space $R_b$. Let
$\mathcal A=R+K$, also an operator on $\mathscr S$. 
Thus, with $P=\xi \xi^t$, we have $H_\xi=\mathcal{A}P$, and
\begin{equation}
 Q(\xi)=Q_{c,K}(\xi)=2\|\mathsf S_\xi(\mathcal AP)\|^2+
 2\operatorname{Re}\langle\mathsf K(C_\xi,P),
                  \mathsf K(C_\xi,\mathcal AP)\rangle
 -2\|Z_2(\xi,C_\xi)\|^2.                                      \label{III:eq:A1}
\end{equation}
For $z\in W$, let $\xi_z \in W^*$ be the element with $\xi_z(u)=\langle u, z\rangle$.

We first prove some propositions that cover all nonzero
finite-dimensional commutative associative algebras.
\begin{proposition}\label{prop:annihilator}
  Let  $\operatorname{Ann}(W)=\{z\in W:c(z,v)=0\text{ for every }v\in W\}$ be the annihilator for the commutative associative multiplication $c$. For every unit vector $z\in \operatorname{Ann}(W)$, let
$B_z(u,v)=\langle c(u,v),z\rangle$. Assume that $Q(\xi_z)=0$, then the restriction of $B_z$ to $z^\perp$ is zero.
\end{proposition}

\begin{proof}
Assume \(B_z|_{z^\perp}\neq0\). Choose an orthonormal basis with \(z=e_m\),
and write \(B=B_z\) on \(z^\perp\). Recall that for $\xi\in W^*$, $C_\xi$ is the symmetric matrix of the bilinear form
\[
(u,v)\mapsto \xi(c(u,v)),\quad u,v\in W,
\]
which yields a linear map
\[
C:W^*\to\mathscr S=\operatorname{Sym}^2W^*,\quad \xi\mapsto C_\xi.
\]
Since $z\in \operatorname{Ann}(W)$, in the orthonormal basis  $\xi_z=e_m^*$, we have $C_{\xi_z}:=C(\xi_z)=B\oplus0$. Also $C^*(P)=0$, for $P=\xi_z\xi_z^t$. Then we compute $D_cP$.
By definition,
\[
(D_cP)_{i,am}
=
\sum_r c_{im}^{\,r}P_{ar}
-\sum_r c_{ia}^{\,r}P_{mr}.
\]
Note that $P_{mm}=1$ is the only nonzero entry of $P$. Therefore, for $i,a<m$
\[
(D_cP)_{i,am}=-c_{ia}^{\,m}=-B_{ia},
\]
since \(B_{ia}=\langle c(e_i,e_a),z\rangle=c_{ia}^{\,m}\).
All other entries of \(D_cP\) vanish, as $c^r_{im}=0$ for any $i,r$. Since $C^*(P)=0$, we have \[H:=(D_c^*D_c+K)P=\mathcal{A}P,\] where $K$ is skew-Hermitian as in Proposition \ref{III:prop:exact-Q1}. Then we know the entry
\[
H_{mm}=\langle H, P\rangle=\langle (D_c^*D_c+K)P,P\rangle
=\langle D_c^*D_cP,P\rangle+\langle KP,P\rangle.
\]
Because $K^*=-K$, the term $\langle KP,P\rangle$ is purely imaginary. Hence
\[
\operatorname{Re}H_{mm}
=\langle D_c^*D_cP,P\rangle
=\|D_cP\|^2=\|B\|^2,
\]
where the last equality follows as $D_cP$ has nonzero entries $-B_{ia}$ for $i,a<m$.

In wedge coordinates $(i,m),(j,m)$, a direct calculation yields 
\[\mathsf K(C_{\xi_z},P)_{im,jm}=B_{ij}, \quad\mathsf K(C_{\xi_z},H)_{im,jm}=H_{mm}B_{ij}.\]
Also the nonzero entries of $Z_2(\xi_z,C_{\xi_z})$ are, up to signs, the
\(2\times2\) minors of \(B\). Hence \eqref{III:eq:A1} yields
\[
Q(\xi_z)=2\bigl(\|B\|^4-\|C_2(B)\|^2\bigr)+2\|\mathsf S_{\xi_z}H\|^2.
\]
If \(s_i\) are the singular values of \(B\), then
\[
\|B\|^4-\|C_2(B)\|^2
=\sum_i s_i^4+\sum_{i<j}s_i^2s_j^2>0.
\]
Thus $Q(\xi_z)>0$, which contradicts the assumption. 
\end{proof}

\begin{definition}
\label{def:character}
    A nonzero character is a nonzero complex linear functional
$\lambda:W\to\C$ with \[\lambda(c(u,v))=\lambda(u)\lambda(v).\]
For $x\in W$, $\lambda_x\in W^*$ is the element with $\lambda_x(u)=x^tu$.  $x$ is called a character vector if $\lambda_x$ is a character.
\end{definition}
For any $x,y\in W$, direct polarization yields the identities
\begin{equation}
 D_c(\lambda_x\lambda_x^t)=\mathsf S_{\lambda_x}(C_{\lambda_x}),\quad
 D_c(\lambda_x\lambda_y^t+\lambda_y\lambda_x^t)=\mathsf S_{\lambda_x}(C_{\lambda_y})+\mathsf S_{\lambda_y}(C_{\lambda_x}).                         \label{III:eq:A2}
\end{equation}
 We can see that $x$ is a character vector if and only if
$C_{\lambda_x}=\lambda_x\lambda_x^t$. 

\begin{lemma}\label{lem:SG-zero}
   For $G\in \mathscr S, x\neq 0$, if $\mathsf S_{\lambda_x}G=0$, then $G\in\C \lambda_x\lambda_x^t$.
\end{lemma}
\begin{proof}
$S_{\lambda_x}G=0$  makes each row of $G$ a multiple of $\lambda_x^t$.
Writing $G=v\lambda_x^t$, symmetry implies that
$v$ is a multiple of $\lambda_x$.
\end{proof}

For a polynomial in a complex scalar $t$ and its conjugate,
$[t^a\bar t^b]$ denotes its coefficient of $t^a\bar t^b$.
\begin{proposition}\label{prop:no_independent_character}
If $Q(\xi)=0$ for every $\xi\in W^*$, then there are no linearly independent character vectors $x,y\in W$.
\end{proposition}
\begin{proof}
Suppose $x,y\in W$ are two linearly independent character vectors.
Choose unitary coordinates so that
\[
 x=\sqrt a\,e_1,\quad y=u e_1+v e_2,\quad
 a>0,\ v>0,\quad \delta=av^2>0,
\]
and let $P_0:=\lambda_x\lambda_x^t=C_{\lambda_x}$, $P_1:=\lambda_x\lambda_y^t+\lambda_y\lambda_x^t$, $P_2:=\lambda_y\lambda_y^t=C_{\lambda_y}$.
Let $E$ be the matrix on $\Lambda^2W^*$ with the only nonzero entry
$E_{12,12}=1$. A straightforward computation yields
\begin{equation}
 \mathsf K(P_0,P_2)=\delta E,\quad
 \mathsf K(P_1,P_1)=-2\delta E,\quad
 \mathsf K(P_0,P_1)=\mathsf K(P_1,P_2)=0,                \label{III:eq:A4}
\end{equation}
and also
\begin{equation}
    \mathsf K(P_0,P_0)=\mathsf K(P_2,P_2)=0, Z_2(\lambda_x,C_{\lambda_x})=Z_2(\lambda_y,C_{\lambda_y})=0, \mathsf S_{\lambda_y}(P_0)=-\mathsf S_{\lambda_x}(P_1).\label{III:eq:A4-1}
\end{equation}
Then from \eqref{III:eq:A1}, $Q(\lambda_x)=0$ reduces to
$\mathsf S_{\lambda_x}(\mathcal AP_0)=0$, so $\mathcal AP_0=\alpha P_0$ by Lemma \ref{lem:SG-zero}.
Similarly $\mathcal AP_2=\beta P_2$.

Set \[H_1=\mathcal AP_1, \quad h_{ij}=(H_1)_{ij}, \quad P=(\lambda_x+t\lambda_y)(\lambda_x+t\lambda_y)^t=P_0+tP_1+t^2P_2, \quad G=H_1-\alpha P_1.\]  A computation yields $\mathsf S_{\lambda_y}(P_0)=-\mathsf S_{\lambda_x}(P_1)$, then the coefficient of $t$ in
$\mathsf S_{\lambda_x+t\lambda_y}(\mathcal AP)$ is 
\begin{equation}
    S_{\lambda_y}(\mathcal{A}P_0)+S_{\lambda_x}(\mathcal{A}P_1)=\alpha S_{\lambda_y}(P_0)+S_{\lambda_x}(\mathcal{A}P_1)= S_{\lambda_x}(\mathcal{A}P_1-\alpha P_1)=\mathsf S_{\lambda_x}G. \label{equ:coeffient-in-S}
\end{equation}
The constant coefficients of $\mathsf K(C_{\lambda_x+t\lambda_y},P), \mathsf K(C_{\lambda_x+t\lambda_y},\mathcal{A}P)$ in \eqref{III:eq:A1} are $\mathsf K(P_0,P_0)=0$ and $\mathsf K(P_0,\mathcal{A}P_0)=\mathsf K(P_0,\alpha P_0)=0$, respectively.
Their coefficients of $t$ are respectively
$\mathsf K(P_2,P_0)+\mathsf K(P_0,P_1)=\delta E$ and $\mathsf K(P_0,H_1)+\alpha\delta E$.
Moreover \[Z_2(\lambda_x+t\lambda_y,C_{\lambda_x+t\lambda_y})=0,\] since all relevant columns lie in
$\operatorname{span}\{\lambda_x,\lambda_y\}$.
Thus
\begin{align*}
[t\bar t]\,
2\operatorname{Re}\langle\mathsf K(C_{\lambda_x+t\lambda_y},P),\mathsf K(C_{\lambda_x+t\lambda_y}, \mathcal{A}P)\rangle
&=
\langle \delta E,\mathsf K(P_0,H_1)+\alpha\delta E\rangle
+
\langle \mathsf K(P_0,H_1)+\alpha\delta E,\delta E\rangle.
\\&=
2\delta\operatorname{Re}\langle E,\mathsf K(P_0,H_1)\rangle
+
2\delta^2\operatorname{Re}\alpha,
\end{align*}
as $\langle E,E\rangle=1$.
Because \(E\) has its only nonzero entry at \((12,12)\),
\[
\langle E,\mathsf K(P_0,H_1)\rangle
=
\overline{\mathsf K(P_0,H_1)_{12,12}}=a\overline{h_{22}}.
\]
We therefore obtain
\begin{equation}
 [t\bar t]Q(\lambda_x+t\lambda_y)
 =2\|\mathsf S_{\lambda_x}G\|^2+2a\delta\operatorname{Re}h_{22}
                  +2\delta^2\operatorname{Re}\alpha.     \label{III:eq:A5}
\end{equation}
For a symmetric matrix $J$, define explicitly
\[
 \mathcal E(J)=
 \sum_{\substack{1\leq i\leq m,\ 2\leq k\leq m\\
                    (i,k)\ne(1,2),(2,2)}}|J_{ik}|^2\geq0.
\]
Since $x=\sqrt{a}e_1$, a straightforward computation yields $\|\mathsf S_{\lambda_x}J\|^2=a\sum_{i=1}^m\sum_{k=2}^m|J_{ik}|^2$.
As $G_{12}=h_{12}-\alpha\sqrt\delta,
G_{22}=h_{22}$, \eqref{III:eq:A5} is exactly
\begin{align}
 [t\bar t]Q(\lambda_x+t\lambda_y)
 ={}&2a|h_{12}-\alpha\sqrt\delta|^2
    +2a|h_{22}+\delta/2|^2+2a\mathcal E(G)\notag\\
   &+\delta^2(2\operatorname{Re}\alpha-\frac a 2),              \label{III:eq:A6}
\end{align}
which is vanishing. As $x,y$ are character vectors, by \eqref{III:eq:A2}, $D_cP_0=D_cP_2=0$. 
Since $\mathcal A P_0=\alpha P_0$, taking inner products with $P_0$ yields
\[
\alpha\|P_0\|^2=\langle \mathcal A P_0,P_0\rangle
=\|D_cP_0\|^2+\frac14\|C^*P_0\|^2+\langle KP_0,P_0\rangle.
\]
The term $\langle KP_0,P_0\rangle$ is purely imaginary, so
\[
(\operatorname{Re}\alpha)\|P_0\|^2
=\|D_cP_0\|^2+\frac14\|C^*P_0\|^2.
\]
Note that $\|P_0\|^2=a^2$,
and $D_cP_0=0$. We have
\[
 \operatorname{Re}\alpha=\frac{\|C^*P_0\|^2}{4a^2}
 .\]
Also, since 
 \[
 \quad \langle C^*P_0,\lambda_x\rangle=\langle P_0, C_{\lambda_x}\rangle=\langle P_0,P_0\rangle = a^2.
\]
Cauchy--Schwarz implies that $\operatorname{Re}\alpha \geq \frac{a}{4}$. Vanishing of \eqref{III:eq:A6} forces equality in Cauchy--Schwarz and hence
\begin{equation}\label{equ:CP0}
    C^*P_0=a\lambda_x.
\end{equation}
 Interchanging $x,y$ and writing $b=\|y\|^2$ yields
$C^*P_2=b\lambda_y$. Moreover, we have
\[a\langle \lambda_x,\lambda_y\rangle=\langle C^*P_0, \lambda_y \rangle=\langle P_0,P_2  \rangle=\langle \lambda_x,\lambda_y\rangle^2,\]
 and similarly
\[b\langle\lambda_x,\lambda_y\rangle=\langle \lambda_x,\lambda_y\rangle^2. \]
A nonzero inner product would force $a=b=\langle \lambda_x,\lambda_y\rangle$ and 
$x,y$ to be linearly dependent by Cauchy--Schwarz. Therefore \[\langle \lambda_x,\lambda_y\rangle=0, u=0.\]
Since $D_cP_2=0, C^*P_2=b\lambda_y$,  \[RP_2=(D_c^*D_c+\tfrac14CC^*)P_2=\frac{b}{4}C(\lambda_y)=\frac{b}{4}P_2,\] 
while $\mathcal AP_2=\beta P_2$; hence $KP_2=\mathcal AP_2-RP_2$ is
proportional to $P_2$ and $K$ preserves $P_2^\perp$, as $K$ is skew-Hermitian. As $u=0$ and $y=ve_2$, $\langle P_1, P_2\rangle=0$, and hence
$\langle KP_1,P_2\rangle=0$.
Also, note that $\langle D^*_cD_cP_1,P_2\rangle=\langle D_cP_1,D_cP_2 \rangle=0$, we have
\[
 \langle RP_1,P_2\rangle=\frac{1}{4}\langle CC^*P_1,P_2 \rangle
 =\tfrac14\langle C^*P_1,b\lambda_y\rangle
 =\tfrac b4\langle P_1,P_2\rangle=0.
\]
 It follows that \[\langle H_1, P_2\rangle=\langle \mathcal{A}P_1, P_2\rangle=\langle RP_1+KP_1, P_2\rangle=0.\] Again, as $u=0$ and $y=ve_2$, this implies $h_{22}=0$, contradicting the
requirement $h_{22}=-\delta/2$ in \eqref{III:eq:A6}.
\end{proof}

A \textit{derivation} at the character $\lambda_x$ is a complex linear functional
$\lambda_y$ satisfying \[\lambda_y(c(v,w))=\lambda_x(v)\lambda_y(w)+\lambda_y(v)\lambda_x(w),\] in term of matrices, this is equivalent to
\[C_{\lambda_y}=\lambda_x\lambda_y^t+\lambda_y\lambda_x^t.\] 
\begin{proposition}\label{prop:no_independent_derivation}
If $Q(\xi)=0$ for every $\xi\in W^*$, then there are no linearly independent vectors $x,y\in W$ such that $\lambda_x$ is a character and $\lambda_y$ is either a derivation at $\lambda_x$, or satisfies $C_{\lambda_y}=0$.
\end{proposition}

\begin{proof}
Suppose that such linearly independent $x,y$ exist.  Retain the notation in Proposition \ref{prop:no_independent_character} unless otherwise specified. Now we have
\[
 C_{\lambda_x}=P_0,\quad C_{\lambda_y}=P_1\ \hbox{or}\ C_{\lambda_y}=0.
\]
Again, $Q(\lambda_x)=0$ yields $\mathcal AP_0=\alpha P_0$. Also, $P_2=\lambda_y\lambda_y^t$ (not necessarily equals $C_{\lambda_y}$ as $y$ is not assumed to be a character vector). Similar as in Proposition \ref{prop:no_independent_character}, the constant coefficients of $\mathsf K(C_{\lambda_x+t\lambda_y},P), \mathsf K(C_{\lambda_x+t\lambda_y},\mathcal{A}P)$ in \eqref{III:eq:A1} are $\mathsf K(P_0,P_0)=0$ and $\mathsf K(P_0,\mathcal{A}P_0)=\mathsf K(P_0,\alpha P_0)=0$, respectively.
The coefficient of $t$ in $\mathsf K(C_{\lambda_x+t\lambda_y},P)$ is 
\[\mathsf K(C_{\lambda_y},P_0)+\mathsf K(P_0,P_1)=0,
\] from \eqref{III:eq:A4}, no matter whether $C_{\lambda_y}=P_1\ \hbox{or}\ C_{\lambda_y}=0$.
Then we know the term $2\operatorname{Re}\langle\mathsf K(C_{\lambda_x+t\lambda_y},P), \mathsf K(C_{\lambda_x+t\lambda_y},\mathcal{A}P)\rangle$ in \eqref{III:eq:A1} has no $t\bar t$ coefficient.
Similar as in Proposition \ref{prop:no_independent_character}, $Z_2$-term vanishes identically, then from \eqref{equ:coeffient-in-S},
\[
 [t\bar t]Q(\lambda_x+t\lambda_y)=2\|\mathsf S_{\lambda_x}(\mathcal AP_1-\alpha P_1)\|^2.
\]
Its vanishing implies
\begin{equation}
 \mathcal AP_1=\alpha P_1+\eta P_0                     \label{III:eq:A7}
\end{equation}
for some $\eta\in \mathbb C$, by Lemma \ref{lem:SG-zero}. Let \[H_2=\mathcal AP_2, \quad h_{ij}=(H_2)_{ij}, \quad
G=H_2-\eta P_1-\alpha P_2,\] and let $\sigma=-1$ when $C_{\lambda_y}=P_1$,
$\sigma=1$ when $C_{\lambda_y}=0$.

A straightforward computation yields
$
\mathsf S_{\lambda_y}(P_0)=-\mathsf S_{\lambda_x}(P_1),
\mathsf S_{\lambda_y}(P_1)=-\mathsf S_{\lambda_x}(P_2).
$
Hence the coefficient of $t^2$ in the factor
$\mathsf S_{\lambda_x+t\lambda_y}(\mathcal AP)$ is
\[
\mathsf S_{\lambda_x}(H_2)+\mathsf S_{\lambda_y}(\mathcal AP_1)
=
\mathsf S_{\lambda_x}(H_2-\eta P_1-\alpha P_2)=S_{\lambda_x}G.
\]
By \eqref{III:eq:A4},
\[
\mathsf K(C_{\lambda_x+t\lambda_y},P)=\sigma\delta t^2E,
\]
and the coefficient of \(t^2\) in $\mathsf K(C_{\lambda_x+t\lambda_y}, \mathcal A P)$ is
\[
\mathsf K(C_{\lambda_x}, \mathcal{A}P_2)+\mathsf K(C_{\lambda_y},\mathcal AP_1 )=\mathsf K(P_0,H_2)+(\sigma-1)\alpha\delta E.
\]
Because \(E\) has its only nonzero entry at \((12,12)\),
\[
\langle E,\mathsf K(P_0,H_2)\rangle
=
\overline{\mathsf K(P_0,H_2)_{12,12}}=a\overline{h_{22}}.
\]
Also note that $\langle E,E\rangle=1$, we have
\begin{align}
 [t^2\bar t^2]Q(\lambda_x+t\lambda_y)
 =2\|\mathsf S_{\lambda_x}G\|^2+2\sigma a\delta\operatorname{Re}h_{22}
             +2(1-\sigma)\delta^2\operatorname{Re}\alpha.
                                \label{III:eq:A8}
\end{align}
As
$
G_{12}=h_{12}-\alpha uv-\eta\sqrt a\,v,
G_{22}=h_{22}-\alpha v^2, \delta=av^2,
$
we have
\begin{align}
 [t^2\bar t^2]Q(\lambda_x+t\lambda_y)
 ={}&2a|h_{12}-\alpha uv-\eta\sqrt a\,v|^2\notag\\
 &+2a|h_{22}-\alpha v^2+\frac{\sigma\delta}{2}|^2+2a\mathcal E(G)\notag\\
 &+\delta^2(2\operatorname{Re}\alpha-\frac{a}{2}).                  \label{III:eq:A9}
\end{align}
Let $J$ be the orthogonal projection of $P_1$ to $P^\perp_0$.
Then $J=\sqrt a\,v(E_{12}+E_{21})$, $\|J\|^2=2\delta$, and pairing with $J$ on both sides of \eqref{III:eq:A7}
yields
\[
\langle \mathcal A P_1,J\rangle=\alpha\langle P_1,J\rangle=\alpha\|J\|^2.
\]
Also, note that $\mathcal A P_0=\alpha P_0$ and $\langle P_0, J\rangle=0$, we have \[\langle \mathcal A P_1,J\rangle=\langle\mathcal{A}J ,J\rangle.\]
As $K$ is skew-Hermitian, we then obtain
\[
\operatorname{Re}\langle \mathcal A P_1,J\rangle=\langle R
J,J\rangle.
\]
Hence
\[
 \operatorname{Re}\alpha=\frac{\langle RJ,J\rangle}{2\delta} = \frac{\|C^*J\|^2}{8\delta}+\frac{\|D_cJ\|^2}{2\delta}.
\]
When $\lambda_y$ is a derivation at $\lambda_x$, \[\langle C^*J,\lambda_x \rangle=\langle J,P_0\rangle=0, \langle C^*J,\lambda_y \rangle=\langle J,P_1\rangle=\|J\|^2=2\delta.\]
Then, taking the component $v e_2$ of $y$ perpendicular
to $x$, we have \[\langle C^*J,\lambda_{ve_2}\rangle=2\delta.\] 
Applying Cauchy--Schwarz yields \[\|C^*J\|^2\geq4a\delta.\]
When $C_{\lambda_y}=0$, \eqref{III:eq:A2} instead yields $D_c(P_0)=0, D_c(P_1)=S_{\lambda_y}(P_0)$, thus
$D_cJ=\mathsf S_{\lambda_y}(P_0)$. 
A direct calculation yields $\|D_cJ\|^2=\|\mathsf S_{\lambda_y}(P_0)\|^2=a\delta$.
Both cases imply $\operatorname{Re}\alpha\geq \frac{a}{2}$.
The last term of \eqref{III:eq:A9} is therefore at least $\frac{a\delta^2}{2}>0$, a contradiction.
\end{proof}

We include the finite-dimensional algebra facts needed here. Let $B$ be a finite-dimensional commutative unital complex algebra,
and let $L_x:B\to B$, $L_x(v)=xv$, denote the multiplication operators.
\begin{lemma}\label{lem:simultaneous-triangular}
The family $\{L_x\}_{x\in B}$ is simultaneously triangularizable: there exists a basis of $B$ in which every $L_x$ is upper triangular.
\end{lemma}
\begin{proof}
We argue by induction on $n=\dim_{\mathbb C}B$.
If $n=0$, there is nothing to prove. Assume $n\geq 1$ and that the assertion holds for all commutative unital complex algebras of dimension $<n$. If $L_x=\mu_xI$ for some scalar $\mu_x$ for all $x\in B$, then we are done. If not, choose $x_0\in B$ such that $L_{x_0}$
has an eigenvalue $\lambda_0$, with eigenspace
\[
E_0=\ker(L_{x_0}-\lambda_0 I) \subsetneq B.
\]
For every $x\in B$, the operators $L_x$ and $L_{x_0}$ commute, so
$L_x(E_0)\subset E_0$ for every $x\in B$. 
Restrict each $L_x$ to $E_0$:
\[
L_x|_{E_0}:E_0\to E_0.
\]
These restrictions still commute. If \(\dim E_0=1\), then any nonzero
vector in $E_0$ is a common eigenvector. Otherwise, repeat the same
argument on $E_0$, choosing another operator and taking the
eigenspace of its restriction.
Thus we can find $v$ as a common eigenvector of the whole commuting family.
Then $B/\mathbb C v$ is a commutative unital complex algebra of
dimension $n-1$, and the induced operators
\[
\overline{L_x}:B/\mathbb C v\to B/\mathbb C v
\]
are commuting multiplication operators.

By the induction hypothesis, $B/\mathbb C v$ has a basis
$\bar v_2,\dots,\bar v_n$ in which every $\overline{L_x}$ is upper
triangular. Lifting these vectors to $v_2,\dots,v_n\in B$, the basis
$v,v_2,\dots,v_n$ of $B$ puts every $L_x$ in upper triangular
form. This completes the induction.
\end{proof}
From Lemma \ref{lem:simultaneous-triangular}, each $L_x(x\in B)$ is upper triangular by choice of basis. Then
their diagonal entries define finitely many characters
$\lambda_1,\ldots,\lambda_r$ after the repetitions are removed.

\begin{lemma}
   $\lambda_1,\ldots,\lambda_r$ from the diagonal entries include every nonzero character $\lambda$ of $B$.
\end{lemma}
\begin{proof}
Note that for invertible $L_x$, $L^{-1}_x(1)$ yields the inverse of $x$. Thus  $x$ is invertible if and only if $L_x$ is invertible.   For any character $\lambda:B\to\mathbb C$ and any $h\in B$, the element
$
h-\lambda(h)1
$
cannot be invertible, since applying \(\lambda\) would yield 
\[
\lambda(h-\lambda(h)1)=0.
\]
Thus $L_h-\lambda(h)I$ is not invertible, so $\lambda(h)$ is an
eigenvalue of $L_h$. Then $B$ is the union of the kernels of the finitely many linear
forms $\lambda-\lambda_i$. A complex vector space cannot be a
finite union of proper linear subspaces, so one of these forms is
zero.
\end{proof}

The following proposition is a standard result in commutative algebra, see for example \cite[Theorem 8.7]{AM69}.
\begin{proposition}\label{prop:algebra-decompose}
Let $B$ be a finite-dimensional commutative unital complex algebra.
Then $B\cong \prod_{i=1}^r B_i$, where each $B_i$ is a local algebra
whose character kernel is nilpotent.
\end{proposition}

\begin{proof}
Let $\lambda_1, \cdots, \lambda_r$ be the nonzero characters. Choose $h\in B$ outside the finitely many hyperplanes
$\lambda_i(h)=\lambda_j(h)$ for distinct characters, such that $\lambda_1,\dots,\lambda_r$ take distinct values
on $h$. Then the generalized eigenspaces of \(L_h\),
\[
B_i=\ker(L_h-\lambda_i(h)I)^{m_i},
\]
are ideals, and every $L_x$ preserves each $B_i$.  Since the polynomials $(t-\lambda_i(h))^{m_i}$ are pairwise
relatively prime, the Chinese Remainder theorem implies that the
spectral projections are polynomials in $L_h$. Thus the $i$-th spectral projection is $P_i=p_i(L_h)=L_{p_i(h)}$ for some polynomial $p_i$. Hence they are multiplication by
idempotents $e_i=p_i(h)\in B$, with
\[
e_i^2=e_i,\quad e_ie_j=0\ (i\neq j),\quad \sum_i e_i=1.
\]
Thus, we have a decomposition of the algebra $B$ as 
\[
B=\prod_i B_i,\quad B_i=e_iB.
\]
Each $B_i$ has a unique character $\lambda_i$ with $\lambda_i(e_j)=\delta_{ij}$, and its kernel is
$\mathfrak m_i=\ker\lambda_i$. Thus $B_i$ is local. In a simultaneous triangularization,
every element of $\mathfrak m_i$ acts as a strictly upper triangular
operator on $B_i$. Hence products of $\dim B_i$ such operators
vanish. Since multiplication is faithful, $\mathfrak m_i$ is
nilpotent. Therefore each $B_i$ is a local algebra with nilpotent character
kernel.
\end{proof}

Now we are ready for the proof of Proposition \ref{III:lem:algebra-obstruction}.
\begin{proof}[Proof of Proposition \ref{III:lem:algebra-obstruction}]
We first define the unitization $\widetilde{W}=\C1\oplus W$, whose product is
\[
(\alpha1+u)(\beta1+v)
=\alpha\beta1+\alpha v+\beta u+c(u,v).
\]
Its augmentation \[\epsilon:\alpha1+u\mapsto\alpha\] is a character. Every
nonzero character $\lambda$ of $W$ extends uniquely to
\[\alpha1+u\mapsto\alpha+\lambda(u).\] The augmentation
character restricts to zero on $W$, and its other characters correspond
exactly to the nonzero characters of $W$. 
From Proposition \ref{prop:no_independent_character}, $W$ has at most one nonzero character.

If there are none, the augmentation is the only character of the
unitization, then from Proposition \ref{prop:algebra-decompose}, its kernel $W$ is nilpotent. We write $W^r$ for
the span of all $r$-fold products in $W$. Assuming $c\neq0$,
its last nonzero power $W^r$, with $r\geq2$, lies in
\[
\operatorname{Ann}(W)\cap W^2.
\]
For a unit vector $z\in W^r$, write $u=u'+\alpha z,v=v'+\beta z$ with $u',v'\in z^\perp$, then
\[c(u,v)=c(u',v').\]
Since $z\in W^2$, write \begin{equation}\label{equ:expressionW2}
    z=\sum_{i}c(u_i,v_i)=\sum_ic(u'_i,v'_i),
\end{equation} where $u'_i,v'_i$ denote their orthogonal projection to $z^{\perp}$. Recall that $B_z(u,v)=\langle c(u,v),z\rangle$. Then pairing with $z$ on both sides of \eqref{equ:expressionW2}  yields $B_z|_{z^\perp}\neq0$, which contradicts Proposition \ref{prop:annihilator}.

If $W$ has exactly one nonzero character $\lambda$, idempotent decomposition in Proposition \ref{prop:algebra-decompose}
reads $\widetilde{W}=e\widetilde{W}\oplus (1-e)\widetilde{W} $ as an direct product of local algebras where $e$ satisfies  $e^2=e,\epsilon(e)=0, \lambda(e)=1$. Then $e\in W$ and $e\widetilde{W}=eW$ is a unital algebra with identity $e$.
Also $(1-e)\widetilde W=\mathbb{C}(1-e)\oplus N$, where $N=\{v\in W:ev=0\}$ is nilpotent. Thus we have a decomposition \[W=eW\oplus N.\] Products between the two
summands are zero. 

If $\dim eW\geq2$, its character kernel $\mathfrak m$ is nonzero and
$\mathfrak m/\mathfrak m^2\ne0$ by nilpotence. Write the local algebra $eW=\mathbb{C}e\oplus \mathfrak m $.
A nonzero functional on $\mathfrak m/\mathfrak m^2$, extended by zero on $\mathbb Ce$
and $N$, is a derivation $\partial$ at the character $\lambda$.
To check the Leibniz rule, write  elements of
 $W$ as $a=\alpha e+r+n, b=\beta e+s+n'$ with
$r,s\in\mathfrak m, n,n'\in N$. The product is
\[ab=\alpha\beta e+\alpha s+\beta r+rs+nn'.\] The chosen functional
annihilates $e$, $\mathfrak m^2 $ and $N$. Thus
\[\partial(ab)=\alpha \partial (s)+\beta\partial(r),\]
while $\lambda$ annihilates $\mathfrak{m},N$, $\lambda(e)=1$, then
\[\lambda(a)\partial(b)+\partial(a)\lambda(b)=\alpha \partial (s)+\beta\partial(r).\]
Moreover, we can see the derivation $\partial$ at $\lambda$ has coefficient vector linearly independent of the character vector of $\lambda$, which
contradicts Proposition \ref{prop:no_independent_derivation}.

It remains to handle the case $eW=\mathbb Ce$. If $N^2\ne0$, note that $eN=0$, then similar as the proof of the case where $W$ has no nonzero character, the last nonzero power
of $N$ supplies the annihilator vector, which contradicts Proposition \ref{prop:annihilator}.
If $N^2=0$ and $e$ is not orthogonal to $N$, choose a unit vector $z\in N$ with
$\langle e,z\rangle\ne0$; it is an annihilator. Let $e'$ be the projection of $e$ to $z^\perp$, then $c(e',e')=e$ and $B_z(e',e')=\langle e,z\rangle\neq 0$, so again it contradicts Proposition \ref{prop:annihilator}.
Finally suppose $N^2=0$ and $\langle e, N\rangle=0$.
Choose orthonormal coordinates with $e=\|e\|e_1$ and take any nonzero coefficient vector $y$ in the other coordinates so that $\lambda_y(e)=0$. Such a vector exists
because $m\geq2$. Note that $\lambda(e)=1$ and $\lambda(N)=0$, the character $\lambda$ has character vector $x=e_1/\|e\|$, so $C_{\lambda}=\lambda_x\lambda_x^t$. Also, in this case every
product lies in $\mathbb Ce$,  $C_{\lambda_y}=0$ as $\lambda_y(e)=0$.
The final contradiction is the case $C_{\lambda_y}=0$ of Proposition \ref{prop:no_independent_derivation}.
All possible nonzero multiplications are excluded; hence $c=0$.
\end{proof}

\subsubsection{Type III: the local tangent rigidity}\label{sec:type-III-local-rigidity-done}
We inherit the notation from previous sections unless otherwise specified.
\begin{lemma}\label{lem:cubic_at_b_equal_0}
    At $b=0$, the cubic identity \eqref{eq:cubic-common} at $u=0$ yields
$\mathsf S_\xi(\widetilde{T}_0(\xi \xi^t))=0$ for every $\xi\in W^*$, where $T_0$ is the linear isometry $R_0\rightarrow N_0$ and $\widetilde{T}_0=\nu^{-1}_0T_0\rho_0$ is the coordinate operator. 
\end{lemma}
\begin{proof}
  Let $P=\xi\xi^t$, $A(\xi)=J_0(0,\xi)$. We know $C_3(A(\xi))=0$ as $A(\xi)$ is of rank two. From \eqref{equ:nuspsilonc}, \[ F_2(\xi)=T_{0}C_2(A(\xi))=T_0\rho_{0}(P)=\nu_{0}(\widetilde T_0 P)=\begin{pmatrix}
      0 & 0\\
      0 & \widetilde{T}_0P
  \end{pmatrix}.\]
Also \eqref{equ:rho-eps_c} yields \[R_0=\{\begin{pmatrix}
    H&0\\
    0&0
\end{pmatrix},H\in \mathscr S\},\]
where the expression of the the block matrix corresponds to the splitting of the second-compound indices into $(0,i)$ followed
by $(i,j)$, with all lower indices increasing. Then a  computation yields \begin{align*}
    \widetilde Q_2(A(\xi
),F_2(\xi))&=C_2(A(\xi)+F_2(\xi))-C_2(A(\xi))-C_2(F_2(\xi))\\&=\begin{pmatrix}
    0 &\mathsf S_\xi(\widetilde{T}_0(P))\\
    \mathsf S_\xi(\widetilde{T}_0(P))^t,&0
\end{pmatrix}.
\end{align*}
The cubic identity yields
 \[\|P_{R_{0}^{\perp}}
 \widetilde Q_2(A(\xi),F_2(\xi))\|^2=
\|C_3(A(\xi))\|^2,\]
which implies $\mathsf S_\xi(\widetilde{T}_0(\xi \xi^t))=0$.
\end{proof}

\begin{lemma}\label{lem:G-scalar}
    If $\mathsf S_\xi(G(\xi \xi^t))=0$, then the linear operator 
    $G: \mathscr S \rightarrow \mathscr S$ is scalar.
\end{lemma}
\begin{proof}
    From Lemma \ref{lem:SG-zero}, $G$ sends every rank one
symmetric matrix into its own span. The associated eigenvalue is
continuous on $\C^m\setminus\{0\}$ and takes values in a finite
spectrum, so it is constant. Rank-one matrices span $\Sym_m(\C)$,
therefore $G$ is scalar.
\end{proof}

The following proposition yields the required expansion \eqref{equ:expand-T-second-order}(or \eqref{equ:expand-T-second-order-rep}) in Proposition \ref{III:prop:exact-Q1} and Proposition \eqref{III:prop:exact-Q2}, under a mild condition.
\begin{proposition}\label{prop:satisfy-order-two-expansion}
 Suppose that there are nonzero admissible multiplications $b_j\to0$ which satisfy the cubic identity \eqref{eq:cubic-common}. Write $\widetilde{T}_j$ for the coordinate operator
$\widetilde{T}_{b_j}=\nu_{b_j}^{-1}T_{b_j}\rho_{b_j}$, so that
$\widetilde{T}_j^*G_{N,j}\widetilde{T}_j=G_{R,j}$, where $G_{N,j}, G_{R,j}, \nu_{b_j},\rho_{b_j}$ are defined as in Section \ref{subsec:deformcalcu}. Suppose $\widetilde{T}_j\rightarrow I$ and $\operatorname{tr}\widetilde{T}_j$ is real and positive, then
\[\|\widetilde{T}_j-I\|=O(\|b_j\|^2).\]
\end{proposition}
\begin{proof}
Suppress the subscript and let
\[
\varepsilon=\|b\|,\quad V=\widetilde{T}-I,\quad s=\|V\|,
\]where $\|b\|^2=\sum_{i,j,k}|b^k_{ij}|^2$.
Here $C$ is the map associated with the multiplication
$b$, explicitly $(C_\xi)_{ij}=\sum_kb_{ij}^{\,k}\xi_k$.
Write $D=D_b$ and $E=E_b$. The norms of $C,D,E$ are $O(\varepsilon)$,
$O(\varepsilon)$ and $O(\varepsilon^2)$, respectively. For
$\|\xi\|\le1$, let $P=\xi \xi^t$, $H=\widetilde{T}P$, $\chi=C^*H$, and
\[
J=\tfrac12(\xi \chi^t+\chi \xi^t),\quad
Y=\mathsf S_\xi(VP)-\tfrac12 \mathsf R(\chi,C_\xi),\quad
Z=\mathsf K(C_\xi,P)+\mathsf K(C_\xi,VP),
\]
where $\mathsf R(\chi,C)_{i,jk}=\chi_jC_{ik}-\chi_kC_{ij}$.
With $A=J_b(0,\xi)$ and $F_2=\nu_b(\widetilde{T}P)$, the exact block expansion similar as in the proof of Proposition \ref{III:prop:exact-Q2} yields
\[
\widetilde Q_2(A,F_2)=\begin{pmatrix}J&Y\\Y^t&Z\end{pmatrix}
\quad\text{at }u=0.
\]
Recall that $\rho(H)=\begin{pmatrix}
    -H &DH\\(DH)^t&EH
\end{pmatrix}$. A computation yields
\[
\rho^*(\widetilde Q_2(A,F_2))
=
-J+2D^*Y+E^*Z, \quad
\rho^*\rho=I+2D^*D+E^*E.
\]
Note that $P_R=\rho(\rho^*\rho)^{-1}\rho^*$, thus 
\[
P_R\widetilde Q_2(A,F_2)=
\rho(\rho^*\rho)^{-1}\rho^*(\widetilde Q_2(A,F_2))=\rho(h)\]
for \[h=(I+2D^*D+E^*E)^{-1}(-J+2D^*Y+E^*Z)
  =-J+O(\varepsilon s+\varepsilon^3).\]
All estimates here are uniform for $\|\xi\|\le1$ and small
$\varepsilon,s$.  
As \eqref{equ:polarization-D-c},\eqref{equ:polarization-E-c}, we have
\[
DJ=\tfrac12\mathsf S_\xi(C\chi)+\tfrac12\mathsf R(\chi,C_\xi),\quad
EJ=\tfrac12\mathsf K(C_\xi,C\chi).
\]
Note that $CC^*VP=O(\varepsilon^2s), H=\widetilde{T}P=P+VP$, we have
\begin{align*}
    (W_\perp)_{01}=Y-Dh
&=Y+DJ+O(\varepsilon^2 s+\varepsilon^4)
\\&=\mathsf S_\xi\bigl(VP+\tfrac12CC^*H\bigr)
+O(\varepsilon^2s+\varepsilon^4)
\\&=\mathsf S_\xi\bigl((V+\tfrac12CC^*)P\bigr)
+O(\varepsilon^2 s+\varepsilon^4).
\end{align*}
Also note that $\mathsf K(C_\xi, CC^*VP)=O(\varepsilon^3s)$,
\begin{align*}
(W_\perp)_{11}=Z-Eh
&=Z+EJ+O(\varepsilon^3 s+\varepsilon^5) \\
&=\mathsf K(C_\xi,P)+\mathsf K\bigl(C_\xi,VP+\tfrac12CC^*H\bigr)
+O(\varepsilon^3s+\varepsilon^5).
\\
&=\mathsf K(C_\xi,P)+\mathsf K\bigl(C_\xi,(V+\tfrac12CC^*)P\bigr)
+O(\varepsilon^3s+\varepsilon^5).
\end{align*}
In summary, the blocks of
$W_\perp=\widetilde Q_2(A,F_2)-\rho(h)$ satisfy
\begin{align*}
(W_\perp)_{00}&=O(\varepsilon s+\varepsilon^3),\\
(W_\perp)_{01}
 &=\mathsf S_\xi\bigl((V+\tfrac12CC^*)P\bigr)
   +O(\varepsilon^2s+\varepsilon^4),\\
(W_\perp)_{11}
 &=\mathsf K(C_\xi,P)+\mathsf K\bigl(C_\xi,(V+\tfrac12CC^*)P\bigr)
   +O(\varepsilon^3s+\varepsilon^5).
\end{align*}
As \eqref{equ:C_3_A}, the blocks of $C_3(A)$ are $-\mathsf K(C_\xi,P)$, $Z_2$, $Z_2^t$ and
$C_3(C_\xi)$, of orders $\varepsilon,\varepsilon^2,
\varepsilon^2,\varepsilon^3$. Thus the leading order-two squared
norms $\|K(C_\xi,P)\|^2$ cancel exactly. The diagonal defect of $\eqref{eq:cubic-common}$ $\Delta(\xi)$
therefore obeys
\begin{equation}
\label{equ:delta_expansion_estimiate}
\Delta(\xi)=2\|\mathsf S_\xi(VP)\|^2
 +O(\varepsilon^2s+\varepsilon^4+\varepsilon^2s^2).
\end{equation}
In particular no order-$\varepsilon s$ term occurs.  
Also, as \eqref{equ:GN},\eqref{equ:GR}, we have 
\[G_N=I+\frac{1}{2}CC^*, \quad G_R=I+2D^*D+E^*E.\]
Thus $\widetilde{T}^*G_N\widetilde{T}=G_R$ yields
\[
V+V^*=O(\varepsilon^2+s^2).
\]
Write $V=K_0+H_0$, with $K_0$ skew-Hermitian and $H_0$ Hermitian.
Then $\operatorname{tr}K_0=0$ as each $T_j$ is chosen so that $\operatorname{tr}T_j$ is real.
Also, $\|H_0\|=O(\varepsilon^2+s^2)$.
The linear map
\[
K_0\mapsto\bigl(\xi\mapsto \mathsf S_\xi(K_0(\xi \xi^t))\bigr)
\]
is injective on the real vector space of trace-free skew-Hermitian endomorphisms
of $\mathscr S$ from Lemma \ref{lem:G-scalar}, and trace-free scalars are zero. Hence finite-dimensional
norm equivalence, with $\mathrm{d}\mu$ the normalized rotation-invariant measure on
$\{\xi\in W^*:\|\xi\|=1\}$, yields
\begin{equation}\label{equ:compare_1}
    \int\|\mathsf S_\xi(K_0P)\|^2\,\mathrm{d}\mu(\xi)\ge\kappa\|K_0\|^2
\end{equation}
for some $\kappa>0$ depending only on $m$.
Note that $\|a+b\|^2\geq\frac{1}{2}\|a\|^2-\|b\|^2$, we have
\begin{equation}\label{equ:compare_2}
\|K_0\|^2\geq \frac{1}{2}\|V\|^2-\|H_0\|^2,
\end{equation}
and also
\begin{equation}\label{equ:compare_3}
\int\|\mathsf S_\xi(K_0P)\|^2\,\mathrm{d}\mu(\xi)-\int2\|\mathsf S_\xi(H_0P)\|^2\,\mathrm{d}\mu(\xi)\leq \int 2\|\mathsf S_\xi(VP)\|^2\,\mathrm{d}\mu(\xi).\end{equation}
 Integrating $\Delta(\xi)=0$ on both sides and putting \eqref{equ:delta_expansion_estimiate},\eqref{equ:compare_1},\eqref{equ:compare_2},\eqref{equ:compare_3} together, it follows from $\|H_0\|=O(\varepsilon^2+s^2)$ that
\[
s^2\le C(\varepsilon^2s+\varepsilon^4+s^4+\varepsilon^2s^2).
\]
Since $s,\varepsilon\to0$, the last two terms can be absorbed.
The remaining inequality is $s^2\le C(\varepsilon^2s+\varepsilon^4)$.
Hence $s=O(\varepsilon^2)$.
\end{proof}

\begin{proposition}
\label{III:prop:metric-gap}
Let $m\ge2$. In a neighborhood of the zero multiplication, the
only commutative associative $b$ for which there is a linear isometry
$T_b:R_b\to N_b$ satisfying the cubic identity \eqref{eq:cubic-common} is $b=0$.
\end{proposition}

\begin{proof}
Suppose that there are nonzero admissible multiplications $b_j\to0$ which satisfy the cubic identity \eqref{eq:cubic-common}.
Write $\widetilde{T}_j$ for the coordinate operator
$\widetilde{T}_{b_j}=\nu_{b_j}^{-1}T_{b_j}\rho_{b_j}$, so that
$\widetilde{T}_j^*G_{N,j}\widetilde{T}_j=G_{R,j}$, where $G_{N,j}, G_{R,j}, \nu_{b_j},\rho_{b_j}$ are defined as in Section \ref{subsec:deformcalcu}. By \eqref{equ:GN}, \eqref{equ:GR},  the $\widetilde{T}_j$ are bounded and a subsequence converges to a unitary $\widetilde{T}_0$.

At $b=0$, from Lemma \ref{lem:cubic_at_b_equal_0}, the cubic identity \eqref{eq:cubic-common} at $u=0$ yields
$\mathsf S_\xi(\widetilde{T}_0(\xi \xi^t))=0$ for every $\xi\in W^*$. Thus $\widetilde{T}_0$ is scalar, by Lemma \ref{lem:G-scalar}. Multiplication of $\widetilde{T}_j$ by a scalar $\lambda_j(|\lambda_j|=1)$ 
preserves the cubic identity, as it preserves $\widetilde{T}_j^*G_{N,j}\widetilde{T}_j=G_{R,j}$, and $\|\widetilde Q_2(A,F_2)\|$.
Replacing $T_j$ by suitable $\lambda_j T_j$, we may assume that
$T_j\to I$ and $\operatorname{tr}T_j$ is real and positive.

Let
\[
\varepsilon_j=\|b_j\|,\quad  c_j=\frac{b_j}{\varepsilon_j},\quad
L_j=\frac{1}{\varepsilon_j^2}(\widetilde{T}_j-I).
\]
From Proposition \ref{prop:satisfy-order-two-expansion}, these are bounded; each $c_j$ is associative and $\|c_j\|=1$.
After a subsequence $c_j\to c$, $L_j\to L$, where $c$ is associative
and $\|c\|=1$. In the following limit formula $C$ denotes the map
associated with the normalized limit multiplication $c$.
Combining Proposition \ref{III:prop:exact-Q1} with Proposition \ref{prop:satisfy-order-two-expansion}, we know
\[
L=D_c^*D_c-\tfrac14CC^*+K,\quad K^*=-K.
\]
For bounded pairs
$(c_j,L_j)$, we obtain the exact cubic defect $\Delta_{\varepsilon_j}(\xi)$ and  divide it by $\varepsilon_j^4$.
Their remainders are uniform, see Remark \ref{expansion_holds_for_sequence}.  
The resulting limit is exactly the coefficient
computed in Proposition~\ref{III:prop:exact-Q2}:
\[
Q_{c,K}(\xi)=0\quad(\xi\in\C^m).
\]
Proposition~\ref{III:lem:algebra-obstruction} now yields $c=0$, contradicting
$\|c\|=1$. 
\end{proof}

\begin{proposition}[Type III local rigidity]
    \label{III:prop:local}
There exists a neighborhood $\mathscr U_{\mathrm{III}}$ of $E_0$ in
$\Gr(d,\Sym_n\C)$ with the following property. If $F$ is a normalized
isometry from $\B^d$ to $D_n^{\mathrm{III}}$ and its tangent space
$E=F'(0)(\C^d)$ belongs to $\mathscr U_{\mathrm{III}}$, then $E$ is
unitarily congruent to $E_0$.
\end{proposition}

\begin{proof}
For each plane $E$, writing the vector of minors in a quadratic monomial basis yields a coefficient matrix. Note that $\dim\Span C_2(E)$ is exactly the rank of the coefficient matrix, which is a lower semicontinuous function with respect to the plane.
Therefore, if $E$ is a plane in a sufficiently small neighborhood of the standard plane $E_0$,
$\dim\Span C_2(E)\geq \dim\Span C_2(E_0)=d^\#$. Then the quadratic identity in Lemma \ref{lem:quadratic-common} forces $\dim\Span C_2(E)= \dim\Span C_2(E_0)$. 

Lemma~\ref{III:lem:hankel} then yields a unitary congruence taking $E$ to
$\cH_b$, with $b$ commutative, associative, and arbitrarily small when
$E$ is sufficiently close to $E_0$. Proposition \ref{III:prop:metric-gap} consequently forces $b=0$.
The resulting plane $\cH_0$ is precisely $E_0$.
\end{proof}

\begin{remark}
    $n\geq 3$ is essential. For $D_2^{III}$, $\dim R=1$ and $R^\perp$ is trivial. Also, $Q_3$ is trivial. Hence the cubic identity \eqref{eq:cubic-common} is trivial and nontrivial deformation of the tangent subspace does exist.
\end{remark}
\begin{remark}
   In the Type IV domain case, local tangent rigidity fails. Indeed, for $D_n^{IV},d=n-1$, normal forms of $(n-1)$-dimensional tangent subspaces can be easily obtained by the conformal structure, see \cite[Proposition 7]{D22}. They are either tangent to a complex unit ball from the Mok embedding, or tangent to a geodesic subdomain $D^{IV}_{n-1}\subset D^{IV}_{n}$. The former forms a closed orbit in $\operatorname{Gr}(n-1,n)$, while the latter forms an open orbit as the complement. Thus local deformation of the former may touch the open orbit and hence local tangent rigidity fails. 
\end{remark}

\subsection{Selecting a fixed normalization}
\label{sec:exact-tangent-assembly}

\begin{proof}[Proof of Proposition~\ref{prop:exact-tangent}]
By Proposition~\ref{prop:boundary-scaling}, choosing equivalent
normalized maps $F_r=\tau_r\circ F\circ\sigma_r$ with
$E_r=\mathrm{d}(F_r)_0(\C^d)\to E_0$, 
when $r\rightarrow 1^-$, $\|P_{E_r}-P_{E_0}\|_{\text{op}} \rightarrow 0$.
Then we apply Proposition~\ref{I:local} for Type I,
Proposition~\ref{II:local-rigidity} for Type II, and
Proposition~\ref{III:prop:local} for Type III: for one 
$r_0<1$ with $1-r_0$ sufficiently small, there is a target automorphism $k$ fixing zero
with $\mathrm{d}k_0(E_{r_0})=E_0$.  

Let $H=k\circ F_{r_0}$ and $A=\mathrm{d}H_0$. The map $A:\C^d\to E_0$
is an isometric linear bijection by \eqref{eq:normal-common}.
Since $A_0$ is also an isometric linear bijection onto $E_0$, the map
\[
V=A_0^*A:\C^d\rightarrow\C^d
\]
is unitary and $A=A_0V$, since 
$V^*V=A^*P_{E_0}A=A^*A=I_d$ and
$A_0V=P_{E_0}A=A$. Therefore the source unitary $u(z)=V^{-1}z$
yields $\mathrm{d}(H\circ u)_0=A_0$. Taking
\[
\tau_0=k\circ \tau_{r_0}, \sigma_0=\sigma_{r_0}\circ u
\]
proves the assertion. 
\end{proof}

\section{Reconstruction of the holomorphic map and completion of the proof}\label{sec:reconstruction}

In this section $G:\B^d\to D$ is an isometry with $G(0)=0$ and
$\mathrm{d}G_0=A_0$, where $A_0$ is given in 
\eqref{eq:standard-inclusions}. Proposition~\ref{prop:exact-tangent}
has produced such a map equivalent to the arbitrary original isometry.
It is known to experts how to recover the holomorphic map from the standard tangent plane, either using the mechanism of Hermitian Jordan triple systems (cf. \cite{UWZ19}) or the framework of Schubert variety (see the discussion in the introduction). 
In particular, Propositions \ref{prop:reconstruction} and Proposition \ref{II:integration} are obtained in \cite{UWZ19} (cf. Theorem 4.5, Lemma 4.6, Lemma 4.10 and Lemma 4.11). 
We include a direct proof in classical matrix coordinates to obtain the explicit formulas.

\subsection{Type I and III cases}

The following result treats Types I and III together. 

\begin{proposition}\label{prop:reconstruction}
Suppose $G(0)=0$ and $\mathrm{d}G_0=A_0$ in Type I or Type III.
Then $G$ is exactly one of the following maps:
\begin{align}
M^I_\lambda(a,x,y)&=
\begin{pmatrix}a&x\\y&\dfrac{\lambda yx}{1+\lambda a}\end{pmatrix},
&& \label{eq:phase-I}\\
M^{III}_\lambda(a,\zeta)&=
\begin{pmatrix}a&\frac{\zeta^t}{\sqrt2}\\ \frac{\zeta}{\sqrt2}&\dfrac{\lambda \zeta \zeta^t}{2(1+\lambda a)}\end{pmatrix},
&&\label{eq:phase-III}
\end{align}
for some constant $\lambda$ with $|\lambda|=1$. 
\end{proposition}
\begin{proof}
Applying \eqref{eq:normal-common} to 
 $A=A_0:\mathbb C^d \rightarrow\mathcal M$ and $E=E_0$, we get 
$A_0^*G(z)=z$. This is equivalent to 
$ A_0^*\bigl(G(z)-A_0z\bigr)=0$ as well as $ G(z)-A_0z\in E_0^\perp.$
Write $G(z)=A_0z+H(z)$ for  $H(z)\in E_0^\perp.$
It thus follows that
$$ G(a, x, y)=\begin{pmatrix}a&x\\y&W(a, x, y)\end{pmatrix}$$
in the case of  Type I, and 
$$G(a, \zeta)=\begin{pmatrix}a& \frac{\zeta^t}{\sqrt2}\\ \frac{\zeta}{\sqrt2}&W(a, \zeta)\end{pmatrix}$$
in the case of Type III. 
Moreover,   letting
$\mathcal W=M_{n-1, m-1}(\C)$ in Type I and  
$\mathcal W=\Sym_{n-1}(\C)$ in Type III,   $W:\B^d\to\mathcal W$ is
holomorphic with $W(0)=0$ and $\mathrm{d}W_0=0$. 
Write $W=\sum_{j\ge2}W_j$ for its homogeneous Taylor expansion and set
$P(z)=yx$ in Type I, $P(z)=\frac{\zeta \zeta^t}{2}$ in Type III, 
to be a matrix-valued quadratic polynomial. 
Let $ \iota_N:\mathcal W\rightarrow E_0^\perp$ be the isometric embedding given by $ B\mapsto
   \begin{pmatrix}0&0\\0&B\end{pmatrix}$. We consider corner minors of $G$ given by $\left(
[C_2(G)]_{\{1,i+1\},\,\{1,j+1\}}
\right)_{i,j}$ and other \(2\times2\) minors are precisely blocks of $C_2(W)$, or mixed minors $x_jW_{ik}-x_kW_{ij}$ and $y_iW_{kj}-y_kW_{ij}$. 
   In both Type I and III cases, it is easy to verify that corner minors of $G$ are given by $aW-P.$
  On the other hand, for the linear map \(A_0z\), since all the lower-right blocks are zero,  
  the corner block of $C_2(A_0z)=-P$,  while all other \(2\times2\) minors vanish.
  Let $ \iota_C:\mathcal W\to\mathcal H_2$ be the linear map defined by
    \(\iota_C(B)\) to be the compound-coordinate matrix whose corner block is \(B\) and whose other entries are zero, for any $B \in \mathcal W$. By direct calculation, it is easy to verify that $ \iota_C$ is also an isometric embedding. 
By (\ref{eq:quadratic-common}), we have 
 $\|Q_2(A_0z)\|^2= \|G_2(z)\|^2$, where $G_2$  is  the quadratic term of $G$.
 Since   $ G_2(z)=\iota_N(W_2(z))$ and $ Q_2(A_0z)=-\iota_C(P(z))$, it follows that $\|W_2(z)\|^2 = \|P(z)\|^2$. 
 Now we verify that the values of \(P\) span \(\mathcal W\).
In Type I, \(P(a,x,y)=yx\) and \(\mathcal W=M_{n-1,m-1}(\mathbb C)\). By choosing
$a=0, y=t e_i, x=t e_j^t$, where \(t>0\) is sufficiently small, 
$P(a,x,y)=t^2E_{ij}$, where \(E_{ij}\) denotes the elementary matrix with a single \(1\) in position \((i,j)\). 
In Type III, $P(a,\zeta)=\frac{\zeta\zeta^t}{2}, \mathcal W=\operatorname{Sym}_{n-1}(\mathbb C).$
Taking \(a=0\) and sufficiently small \(t\), we have $P(0,t e_i)=\frac{t^2}{2}E_{ii}$, and, for \(i<j\), $P(0,t(e_i+e_j))-P(0,t e_i)-P(0,t e_j)=\frac{t^2}{2}(E_{ij}+E_{ji}).$
In both cases, we show that the values of \(P\) span \(\mathcal W\).  
   By Lemma \ref{lem:gram-common}, there exists a unitary matrix $ U:\mathcal W\rightarrow\mathcal W$ such that 
$ W_2(z)=U(P(z))$.

Now, write $G(z)=A_0z+\iota_N(W(z))$ where $W(z)=W_2(z)+W_3(z)+\cdots$, with $W_2(z)=U(P(z))$.
It follows from (\ref{eq:generic-common}) and (\ref{eq:kernel-common}) that 
\begin{equation}\label{kernel}
1-\langle G(z),G(w)\rangle
+\langle Q_2(G(z)),Q_2(G(w))\rangle
-\langle Q_3(G(z)),Q_3(G(w))\rangle+\cdots
=1-\langle z,w\rangle.
\end{equation}
Fix \(z\).  Since $A_0 z$ is orthogonal to $\iota_N (W(w))$, the contribution of
$\langle G(z),G(w)\rangle$ to the quadratic terms in  \(\bar w\) is $\langle W(z),W_2(w)\rangle
=\langle W(z),U(P(w))\rangle.$ The contribution of the remaining terms on the left hand side to the quadratic terms in  \(\bar w\) is from $\langle Q_2(G(z)),Q_2(G(w))$, which is given by 
\begin{equation*}
\begin{split}
&~~~\langle Q_2(G(z)), Q_2(A_0 w)\rangle = \langle Q_2(G(z)),-\iota_C(P(w))\rangle \\
&=-\langle aW(z)-P(z),P(w)\rangle=\langle P(z)-aW(z),P(w)\rangle. 
\end{split}
\end{equation*}
Consequently, $\langle W(z),U(P(w))\rangle =\langle P(z)-aW(z),P(w)\rangle$.
Thus
$ U^*W(z)-P(z)+aW(z)$ is orthogonal to $P(w)$ for every \(w\).
Since  $U^*W(z)-P(z)+aW(z) \in \mathcal W$ and 
the values \(P(w)\) span \(\mathcal W\), it follows that  \(U^*W(z)-P(z)+aW(z)=0\), yielding $(U^*+aI)W=P.$
Multiplying $(U^*+aI)W=P$ by \(U\), we have
$W+aUW=UP.$
In particular, restricting to the hyperplane $a=0$, $W=UP$ and thus $G$ is of degree at most 2.

Finally, we restrict everything to the hyperplane $a=0$ and then compare the bidegree $(3,3)$ terms in (\ref{kernel}). It is easy to see that the only contribution on both sides is from the term $\langle Q_2(G(z)),Q_2(G(w))\rangle$. By classifying all \(2\times2\) minors of \(G\), we see that in the Type I case,  the only degree 3 minors are
$  \det\begin{pmatrix}
  x_j&x_k\\
  W_{ij}&W_{ik}
  \end{pmatrix}
  =x_jW_{ik}-x_kW_{ij}$ for 
$ j<k$ and $  \det\begin{pmatrix}
  y_i&W_{ij}\\
  y_k&W_{kj}
  \end{pmatrix}
  =y_iW_{kj}-y_kW_{ij}$ for
$ i<k$; and $\sum_i \sum_{j<k}|\zeta_jW_{ik}-\zeta_kW_{ij}|^2$ in the Type III case. 
By setting \(w=z\), it follows 
$$\sum_i \sum_{j<k}|x_jW_{ik}-x_kW_{ij}|^2
+
\sum_j \sum_{i<k}|y_iW_{kj}-y_kW_{ij}|^2
=0$$ in the Type I case and 
$$\sum_i \sum_{j<k}|\zeta_jW_{ik}-\zeta_kW_{ij}|^2=0$$
in the Type III case. 
In Type I, we have $x_jW_{ik}-x_kW_{ij}=0$ for $j < k$ and  $ y_iW_{kj}-y_kW_{ij}=0$ for $i < k$. 
Fix \(x\ne0\) and choose \(j_0\) with \(x_{j_0}\ne0\). Then the first identity implies $W_{ik}=\frac{W_{i j_0}}{x_{j_0}}x_k.$
So $W=vx$ for some column \(v\in\mathbb C^{n-1}\). 
Substituting this into the second identities, and using the column \(j_0\), it follows that 
$(y_iv_k-y_kv_i)x_{j_0}=0.$ Hence
$y_iv_k-y_kv_i=0.$
For \(y\ne0\), this implies \(v\in\mathbb Cy\). Consequently,
$U(yx)\in\mathbb C(yx).$ 
In Type III, the similar argument yields $U(\zeta\zeta^t)\in\mathbb C(\zeta\zeta^t)$. 
It follows from  the linear algebra argument that $U=\lambda I_\mathcal W$ for $|\lambda|=1$ since $U$ is unitary.  

Using $(U^*+aI)W=P$, we obtain 
\[
W=\frac{1}{\bar\lambda+a}P=\frac{\lambda}{1+\lambda a}P.
\]
This is \eqref{eq:phase-I} or \eqref{eq:phase-III}, respectively.
\end{proof}

\subsection{Type II case}
\label{II:reconstruction}

Let $n\geq5$ and let $U=\operatorname{span}(e_3,\ldots,e_n)$.
The source ball $\B^d=\B^{2n-3}$ is defined by $|s|^2+\|u\|^2+\|v\|^2<1$.
We follow the exactly the same strategy in the proof as before. 

\begin{proposition}\label{II:integration}
Suppose $G(0)=0$ and $\mathrm{d}G_0=A_0$ in Type II.
Then there is a scalar $\lambda$ with $|\lambda|=1$ such that 
\begin{equation}\label{II:model}
 G(s,u,v)=M^{II}_{\lambda}(s,u,v)
 =s e_1\w e_2+e_1\w u+e_2\w v
       -\frac{\lambda}{1-\lambda s}\,u\w v.
\end{equation}
\end{proposition}

\begin{proof}
By \eqref{eq:normal-common}, we may write
$G(s,u,v)=A_0(s,u,v)+h(s,u,v),$ with $ \mathrm{d}G_0(s,u,v)=A_0(s,u,v)=s\,e_1\wedge e_2+e_1\wedge u+e_2\wedge v, $
where $h\in\Lambda^2U, h=O(2)$. The normal space $E_0^\perp$ is $\Lambda^2U$. 
Let $h_j$ denote the term of total holomorphic degree $j$ in the
Taylor expansion of $h$ at $(s,u,v)=0$. Recall that
 $Q_k(Z)=\frac{1 }{k!}Z^{\w k}\in\Lambda^{2k}\mathbb C^n$ for
$1\leq k\leq\lfloor n/2\rfloor$. Then $Q_2$ splits into
mutually orthogonal  blocks:
\begin{equation}\label{II:q2-decomposition}
 \begin{split}
 Q_2(G)&=e_1\w e_2\w Q+e_1\w R+e_2\w S+\tfrac12h\w h,
 \end{split}
\end{equation}
where $Q=sh-u\w v, R=u\w h, S=v\w h$.
In particular, $Q_2(A_0(s,u,v)) =-e_1\w e_2\w u\w v$.
Define  $J:\Lambda^2U\longrightarrow\Lambda^4\mathbb C^n$,  given by $J(\beta)=e_1\wedge e_2\wedge\beta$.
Thus $Q_2(A_0(s,u,v))=-J(u\wedge v).$
Since \(\Lambda^2U\) is spanned by $u\wedge v$ for $u, v \in U$, $\operatorname{span}\{Q_2(A_0z):z\in\mathbb B^d\} =J(\Lambda^2U) =e_1\wedge e_2\wedge\Lambda^2U$, denoted by $R_0$. It is obvious that $J: \Lambda^2U \to R_0$ is 
an isometric isomorphism. By (\ref{eq:quadratic-common}), there is a linear isometric embedding $T: R_0 \to \Lambda^2U$ such that $h_2(z) = T Q_2(A_0 z)$. 
Thus $\mathcal T= T \circ J: \Lambda^2U\rightarrow\Lambda^2U$ is an unitary transformation on $\Lambda^2U$. In particular, 
$h_2(z) =-{\mathcal T} u \wedge v$
and
$\operatorname{span}\{h_2(s,u,v)\}
={\mathcal T} \bigl(\operatorname{span}\{u\wedge v\}\bigr)
=\Lambda^2U.$

Write $ w=(\sigma,\xi,\eta)$, where \(\sigma\in \C, \xi,\eta\in U\). 
The contribution to the quadratic terms of $\bar w$ in (\ref{kernel}) are merely from $\langle G(z),G(w)\rangle$ is $\langle G(z),h_2(w)\rangle  =\langle h(z),h_2(w)\rangle$ since $h_2(w)$ is orthogonal to $G(z)-h(z) \in E_0$, from $\langle Q_2(G(z)),Q_2(G(w))\rangle$
is $\left\langle Q_2(G(z)),-J(\xi\wedge\eta)\right\rangle =\left\langle J(Q(z)),-J(\xi\wedge\eta)\right\rangle =\left\langle Q(z),-\xi\wedge\eta\right\rangle$.
It follows that $\langle h(z),h_2(w)\rangle =\langle Q(z),-\xi\wedge\eta\rangle.$ Noting \(h_2(w)=-\mathcal T(\xi\wedge\eta)\), we have  $\langle {\mathcal T}^*h(z)-Q(z), \xi\wedge\eta\rangle=0$ for every $\xi\wedge\eta \in \Lambda^2U$. Since ${\mathcal T}^*h(z)-Q(z) \in \Lambda^2U$, it follows that ${\mathcal T}^*h(z)-Q(z) =0$, implying $h=\mathcal TQ=\mathcal T(sh-u\wedge v).$
Comparing  degree 3 terms, we have $h_3=s \mathcal T h_2$. 
Moreover, $h$ satisfies the equation  $(I-s \mathcal T)h=-\mathcal T(u\wedge v).$
Since \(I-s \mathcal T\) is invertible for \(|s|<1\),  we have
\begin{equation}\label{albert}
h(s, u, v)=-(I-s \mathcal T)^{-1} \mathcal T(u\wedge v).
\end{equation}

Noting the linear term of $G(z)$ is $A_0 z= s\,e_1\wedge e_2+e_1\wedge u+e_2\wedge v$, the degree 3 terms in $Q_3(G(z))=\frac1{3!} G(z)\wedge G(z)\wedge G(z)$ must vanish. 
In addition, By (\ref{II:q2-decomposition}),  the degree 3 terms in $Q_2(G(z))$ is given by $e_1\wedge e_2\wedge sh_2 +e_1\wedge u\wedge h_2  +e_2 \wedge v\wedge h_2$. 
By comparing bidegree \((3,3)\) terms in (\ref{kernel}),  we get 
$$\|h_3(z)\|^2 =\|sh_2(z)\|^2 +\|u\wedge h_2(z)\|^2 +\|v\wedge h_2(z)\|^2.$$
Using $\|h_3\|^2 =|s|^2\|\mathcal T h_2\|^2=|s|^2\|h_2\|^2 $ as $\mathcal T$ is unitary, it follows $u\wedge \mathcal T(u\wedge v)=0, v\wedge \mathcal T(u\wedge v)=0$ for all $u, v \in U$. 
It then follows from the linear algebra argument that $\mathcal T =\lambda I_{\Lambda^2U}$ for $|\lambda|=1$ as $\mathcal T$ is unitary. By (\ref{albert}),
$h(s, u, v)=-\frac{\lambda}{1-\lambda s}\,u\w v,$ yielding \eqref{II:model}.
\end{proof}

\subsection{Proof of Theorem \ref{thm:main}}
\label{sec:phase-and-completion}

\begin{lemma}\label{lem:phase-removal}
For each of the three types and each \(\lambda\in\C\) with
\(|\lambda|=1\), there are 
\(\sigma_\lambda \in {\rm Aut}(\B^d)\) and \(\tau_\lambda \in {\rm Aut}(D)\) such that
$ M_\lambda=\tau_\lambda\circ M\circ\sigma_\lambda$,
where \(M=M_1\) is given in
Theorem~\ref{thm:main}.
\end{lemma}

\begin{proof}
If $D$ is of Type I, let $\sigma_\lambda(a,x,y)=(\lambda a,x,\lambda y)$, 
$ \tau_\lambda(Z)=Z\operatorname{diag}(\bar\lambda,I_{m-1}).$
If $D$ is of Type II, let  $\sigma_\lambda(s,u,v)=(\lambda s,\lambda u,\lambda v)$,
$\tau_\lambda(Z)=\lambda^{-1}Z$. If $D$ is of Type III, choose \(\eta\in\C\) with \(\eta^2=\lambda\) and let $\sigma_\lambda(a,\zeta)=(\lambda a,\eta\zeta),$ $\tau_\lambda(Z)=\left(\operatorname{diag}(\eta^{-1},I_{n-1})\right) Z \left({\operatorname{diag}(\eta^{-1},I_{n-1})}\right)^t.$ It is easy to verify that $\sigma_\lambda  \in {\rm Aut}(\B^d), \tau_\lambda \in {\rm Aut}(D)$ and $ M_\lambda=\tau_\lambda\circ M\circ\sigma_\lambda$.
\end{proof}

\begin{proof}[Proof of Theorem~\ref{thm:main}]
Theorem \ref{thm:main} is proved by combining Proposition~\ref{prop:exact-tangent}, Proposition~\ref{prop:reconstruction},   Proposition~\ref{II:integration} and Lemma \ref{lem:phase-removal}.
\end{proof}

{\bf Acknowledgement.} The first author would like to thank Professor Ngaiming Mok for many helpful discussions on this problem and for the encouragement over the years. C. Ding was partially supported by start-up funding from Shenzhen University and Shenzhen Peacock Plan. Y. Yuan was partially supported by Zhejiang
Provincial Natural Science Foundation of China (Grant No. LQKWL26A0201).

\medskip

{\bf AI Statement} This paper is motivated by a desire to understand the methods employed by Yang \cite{Y17} and Mok \cite{M18, M26}. The overall strategy follows Mok’s framework. ChatGPT 6 Astra was used to improve the exposition of the manuscript and to assist in developing some of the mathematical arguments. The authors take full responsibility for the mathematical content of the paper.

\bigskip

\noindent Cong Ding, congding@szu.edu.cn, School of Mathematical Sciences, Shenzhen University, Shenzhen 518060, Guangdong, China

\medskip

\noindent Yuan Yuan, yuanyuan@westlake.edu.cn, Institute for Theoretical Sciences, Westlake University, 
Hangzhou 310024, Zhejiang, China

\end{document}